\documentclass[12pt]{amsart}
\usepackage{etex}
\usepackage{amsfonts,amssymb,graphicx}
\usepackage{youngtab}
\usepackage{rotating}
\usepackage{array}
\usepackage{pdflscape}
\usepackage{tikz}
\usepackage{enumerate}

\usepackage{color}
\newcommand{\highlight}[1]{\colorbox{yellow}{#1}}

\newcommand{\cut}[1]{}

\DeclareMathOperator{\Res}{Res}

\input prepictex
\input pictex
\input postpictex

\numberwithin{equation}{section}

\newtheorem{theorem}{Theorem}[section]
\newtheorem{lemma}[theorem]{Lemma}

\newtheorem{corollary}[theorem]{Corollary}

\newtheorem{question}[theorem]{Question}

\theoremstyle{definition}

\theoremstyle{remark}
\newtheorem{remark}[theorem]{Remark}

\newcommand{\CC}{\mathbb{C}}

\newcommand{\RR}{\mathbb{R}}
\newcommand{\QQ}{\mathbb{Q}}
\newcommand{\ZZ}{\mathbb{Z}}

\newcommand{\OO}{\mathcal{O}}

\newcommand{\la}{\langle}
\newcommand{\ra}{\rangle}

\newcommand{\End}{\mbox{\rm End}}

\newcommand{\cal}{\mathcal}

\begin{document}
\title{Parabolic degeneration of rational Cherednik algebras}
\thanks{SG and AG acknowledge the financial support of Fondecyt grant 1110072.  The work of ML was partially supported by the DFG-Schwerpunkt Grant 1388. DJ was supported  by ANR grant  n. ANR-13-BS01-0001-01. SG and ML thank Arun Ram and the University of Melbourne for providing the very pleasant working environment where part of this research was carried out. DJ and ML are grateful to RIMS for the hospitality, MPI for providing remote memory access to their calculators and Ulrich Thiel for helpful correspondence during the implementation of some examples. We thank Iain Gordon for explaining the proof of Lemma \ref{hwo}.}

\author{Stephen Griffeth}
\address{Instituto de Matem\'atica y F\'isica, Universidad de Talca, Talca, Chile}
\email{sgriffeth@inst-mat.utalca.cl}

\author{Armin Gusenbauer} 
\address{Instituto de Matem\'atica y F\'isica, Universidad de Talca, Talca, Chile}
\email{armingk@inst-mat.utalca.cl}

\author{Daniel Juteau}
\address{Laboratoire de Math\'ematiques Nicolas Oresme,
Universit\'e de Caen Basse-Normandie, CNRS UMR 6139,
14032 Caen,
France}
\email{daniel.juteau@unicaen.fr}

\author{Martina Lanini}
\address{FAU Erlangen-N\"urnberg, Cauerstr. 11, 91058 Erlangen, Germany}
\email{lanini@math.fau.de}

\begin{abstract}
We introduce parabolic degenerations of rational Cherednik algebras of complex reflection groups, and use them to give necessary conditions for finite-dimensionality of an irreducible lowest weight module for the rational Cherednik algebra of a complex reflection group, and for the existence of a non-zero map between two standard modules. The latter condition reproduces and enhances, in the case of the symmetric group, the combinatorics of cores and dominance order, and in general shows that the $c$-ordering on category $\OO_c$ may be replaced by a much coarser ordering. The former gives a new proof of the classification of finite dimensional irreducible modules for the Cherednik algebra of the symmetric group.
\end{abstract}

\maketitle

\section{Introduction}

\subsection{} The purpose of this paper is to introduce a new tool, which we call \emph{parabolic degeneration}, for the study of rational Cherednik algebras. We apply it to give two necessary conditions: first, for the existence of a non-zero map $\Delta_c(\lambda) \rightarrow \Delta_c(\mu)$ between two standard modules, and second, for the finite dimensionality of the irreducible modules $L_c(\lambda)$ in category $\OO_c$. Our explicit computation of corank one parabolic degenerations of standard modules implies our necessary conditions.

A weak version of our criterion for finite dimensionality implies that for the Cherednik algebra of the symmetric group $S_n$ acting on its reflection representation, and with $c>0$, the only possible finite dimensional module is $L_c(\mathrm{triv})$. This fact combined with an analysis of the polynomial representation of the rational Cherednik algebra may be used to give a new proof of the theorem of Berest-Etingof-Ginzburg \cite{BEG} to the effect that the only finite dimensional module is $L_c(\mathrm{triv})$, occurring exactly when $c$ is a positive rational number with denominator $n$ (for negative $c$ it is the sign representation that shows up instead). In fact, our proof of the weak form of our theorem requires so little machinery that we give it here in the introduction.

\subsection{Data} In order to state our necessary conditions we fix some notation; for precise definitions we refer to section \ref{basics}. Let $G$ be a finite subgroup of $\mathrm{GL}(V)$, where $V$ is a finite dimensional $\CC$-vector space. Let
$$R=\{r \in G \ | \ \mathrm{codim}(\mathrm{fix}_V(r))=1 \}$$ be the set of reflections in $G$. The group $G$ is a \emph{reflection group} if it is generated by $R$. For each element $r \in R$ let $c_r \in \CC$ be a complex number, and assume $c_r=c_{g r g^{-1}}$ for all $g \in G$ and $r \in R$; we will often write simply $c$ for the collection of all the numbers $c_r$. These data produce a rational Cherednik algebra $H_c=H_{c}(G,V)$ and a full subcategory $\OO_c$ of $H_c$-mod. Every finite dimensional $H_c(G,V)$-module belongs to $\OO_c$. The most interesting case is for $G$ a reflection group, which we will assume throughout this article unless stated otherwise.

The sum $\sum_{r \in R} c_r (1-r)$ is central in $\CC G$ and, fixing an index set $\Lambda$ for the irreducible $\CC G$-modules, for $\lambda \in \Lambda$ we write $c_\lambda$ for the scalar by which it acts on the irreducible representation $S^\lambda$ of $G$. Define a partial order $<_c$ on $\Lambda$ by $\mu >_c \lambda$ if $c_\mu-c_\lambda \in \ZZ_{>0}$. The category $\OO_c$ is a highest weight category with this ordering and with standard $\Delta_c(\lambda)$ and irreducible $L_c(\lambda)$ objects indexed by $\lambda \in \Lambda$. The finite dimensional irreducible $H_c(G,V)$-modules are, up to isomorphism, contained in the set of $L_c(\lambda)$'s.

\subsection{Maps between standard modules} In this subsection we suppose $\mathrm{dim}(V)=n$ and $G$ acts on $V$ without non-trivial fixed points, $V^G=0$. Define an equivalence relation on the set of points of $V$ by $p \simeq q$ if the stabilizer groups are equal, $G_p=G_q$. The equivalence classes for this relation are called \emph{strata}. Fix a collection of strata $V_j^\circ$ such that the closures in $V$ of the strata $V_j=\overline{V_j^\circ}$ form a complete flag $$V_\bullet=\left( 0=V_0 \subseteq V_1 \subseteq V_2 \subseteq \cdots \subseteq V_n=V \right)$$ of subspaces of $V$. Write $G_j=G_{p}$ for the stabilizer of a point $p \in V_j^\circ$, so that $G_0=G$ and $G_n=\{ 1\}$. For each irreducible representation $S^\lambda$ of $G$ we define a set $\mathrm{SYT}_{V_\bullet}(\lambda)$ as follows: an element $T \in \mathrm{SYT}_{V_\bullet}(\lambda)$ is a sequence $T=(\lambda=\lambda^0,\lambda^1,\dots,\lambda^n)$ where $S^{\lambda^i}$ is an irreducible representation of $G_j$ and $S^{\lambda^0,\lambda^1,\dots,\lambda^n} \neq 0$. Here we recursively define $S^{\lambda^0,\lambda^1,\dots,\lambda^j}$ as the isotypic component,  of isotype $S^{\lambda^j}$, of the restriction to $G_j$ of $S^{\lambda^0,\lambda^1,\dots,\lambda^{j-1}}$. Let $R_j=R \cap G_j$ be the set of reflections in $G_j$ and given $T \in \mathrm{SYT}_{V_\bullet}(\lambda)$, let $c_{T,n-j}$ be the scalar by which $$\sum_{r \in R_{j}} c_r r-\sum_{r \in R_{j+1}} c_r r$$ acts on $S^{\lambda^0,\lambda^1,\dots,\lambda^{j+1}}$.  As an example, if $G=S_{n+1}$, we may choose the flag to correspond to the chain $G_j=S_{n+1-j}$ for $j=0,\dots,n$, and in this case the set $\mathrm{SYT}_{V_\bullet}(\lambda)$ may be put in bijection with standard Young tableaux of shape $\lambda$ (this case motivates our notation), in such a way that $c_{T,k}=\mathrm{ct}(T^{-1}(k+1)) c$ is $c$ times the content of the box containing $k+1$. 

\begin{theorem} \label{standard map1}
If there is a non-zero map $\Delta_c(\lambda) \rightarrow \Delta_c(\mu)$ then for each flag $V_\bullet$ as above there are $T \in \mathrm{SYT}_{V_\bullet}(\lambda)$ and $U \in \mathrm{SYT}_{V_\bullet}(\mu)$ with $c_{U,j}-c_{T,j} \in \ZZ_{\geq 0}$ for all $j=1,2,\dots,n$.
\end{theorem}

This theorem is a corollary of a more precise (and more complicated to apply in practice) recursive version that we will state next. Given a one-dimensional stratum $S$ let $G_S$ be the corresponding maximal parabolic subgroup of $G$ with reflections $R_S=R \cap G_S$, and let $N_S$ be the normalizer of $G_S$ in $G$. The group $N_S$ splits as $N_S=G_S \rtimes \la h \ra$ for an element $h \in G$, and we write $n_S$ for the order of $h$. This splitting is known to hold for real reflection groups thanks to \cite[Lemma 2]{How}, and has been verified as well, case by case, for the irreducible complex reflection groups in the thesis \cite{Mur}. We denote by $L$ the one dimensional representation of $N_S$ factoring through
$N_S / G_S$ on which $h$ acts by $e^{2 \pi i/n_S}$.  Let $S^\perp$ be the $G_S$-stable complement to $\overline{S}$ in $V$. For an irreducible $N_S$-module $S^\chi$, we let $c_\chi$ be the scalar by which $\sum_{r \in R_S} c_r(1-r)$ acts on $S^\chi$. Moreover, given any $G_S$-module $M$ we define an induced module $\Delta_c(M)$ in the same way as for irreducible modules.  
\begin{theorem} \label{standard recursion}
If there is a non-zero map $\Delta_c(\lambda) \rightarrow \Delta_c(\mu)$ then for each one-dimensional stratum $S$ there are irreducible representations $S^\chi$ and $S^\nu$ of $N_S$ with
$(c_\lambda - c_\mu) - (c_\chi - c_\nu) \in \ZZ_{\geq 0}$ and a non-zero map of $H_c(G_S,S^\perp) \rtimes N_S$-modules
\[
\Delta_c(S^{\chi} \otimes_\CC L^{\otimes ((c_\lambda - c_\mu) - (c_\chi - c_\nu))})
\rightarrow \Delta_c(S^{\nu}).
\]
\end{theorem} 

By noting that $(c_\lambda - c_\mu) - (c_\chi - c_\nu)=c_{U,n}-c_{T,n}$ for certain SYT's $U$ and $T$, Theorem \ref{standard map1} follows from this by induction on dimension, forgetting about the action of $h$. Even in the case where $G=S_{n+1}$ is the symmetric group and the stratification is chosen to correspond to the tower $S_{n+1} \supseteq S_n \supseteq S_{n-1} \supseteq \cdots$ these theorems give a stronger result than what we could deduce, for example, from the results of \cite{Gri}. For instance, they imply that when $G=S_4$ and $c=1/3$ the only non-trivial map from a standard module to the polynomial representation has domain $\Delta_{1/3}((2,2))$.  In certain cases the more precise recursive version can be made into a closed form; in this paper we do so, in Theorem \ref{gr1n}, only for the monomial groups $G(r,1,n)$ and the chain $G(r,1,n) \supseteq G(r,1,n-1) \supseteq G(r,1,n-2) \supseteq \cdots$.

As somewhat larger example, suppose $G=S_8$ is the symmetric group acting on its reflection representation, $c=1/4$, and $S^\mu=\mathrm{triv}$ is the trivial one-dimensional representation. For any flag, there is only one possible $U$ as in the conclusion of Theorem \ref{standard map1}, which for the chain $S_8 \supseteq S_7 \supseteq S_6 \supseteq \cdots$ may be identified with the classical standard Young tableau $$U=\young(12345678).$$ For this chain, the possible standard Young tableaux $T$ that can arise are 
$$\young(123,456,78), \ \young(1234567,8), \  \young(1238,456,7).$$ By considering the chain of parabolics beginning $S_8 \supseteq S_6 \times S_2 \supseteq S_6 \supseteq S_5 \supseteq \cdots$ a calculation eliminates the last shape. According to Dunkl's classification \cite{Dun1}, \cite{Dun2}, the remaining two shapes are correspond to precisely the maps with target $\Delta_{1/4}(\mathrm{triv})$. In the case of the symmetric group, our theorem is not always quite so precise, but we do not know of examples except when $c=1/2$ when it fails to identify exactly the homomorphisms between standard modules. We would be surprised if these do not exist, however.

We remark that as a formal consequence of the fact that $\OO_c(G,V)$ is a highest weight category with BGG reciprocity, either one of these theorems implies that we may replace the $c$-ordering on $\OO_c$ by a much coarser ordering preserving the highest weight structure (see Lemma \ref{hwo}). Therefore the transitive closure of the relations in our theorems give necessary conditions for two irreducible objects in $\OO_c$ to be in the same block. Since $\OO_c$ is a highest weight cover of the module category of the finite Hecke algebra, one obtains a necessary condition for two irreducible Hecke algebra modules to belong to the same block. It should be interesting to determine the block structure more precisely; we do not know if the necessary condition for two standard modules to belong to the same block implied by Theorem \ref{standard recursion} is also sufficient (except for classical types, where it is). At least for the real exceptional groups, one can in principle check this against the known classification of blocks for finite Hecke algebras. 

\subsection{Finite dimensionality} Let $v \in V$ be a non-zero vector and let $S=\CC^\times v$ be the set of non-zero multiples of $v$ (we do not assume that $S$ is a stratum). As above, write $G_S$ for the stabilizer of $v$ in $G$ and $N_S=G_S \rtimes \la h \ra$ for the setwise stabilizer of $S$ in $G$, with $h$ of order $n_S$, and let $L$ be the one-dimensional $N_S$ module on which $G_S$ acts trivially and $h$ acts by $e^{2 \pi  i /n_S}$. 

Each $H_c(G,V)$-module $M$ is in particular a $\CC[V]$-module, and we define its \emph{support} $\mathrm{supp}(M)$ to be the zero set of its annihilator in $\CC[V]$. For example, $M$ is finite dimensional if and only if $\mathrm{supp}(M) = \{0\}$. 

\begin{theorem}\label{main}
Suppose $S=\CC^\times v$ for some non-zero vector $v \in V$. If $S \nsubseteq \mathrm{supp}(L_c(\lambda))$, then the restriction $\Res^G_{N_S} S^\lambda$ belongs to the Serre subcategory of $\CC N_S$-mod generated by the modules $L^{\otimes (c_\mu - c_\lambda)} \otimes \Res^G_{N_S} S^\mu$, for those $\mu$ such that $c_\mu - c_\lambda \in \ZZ_{> 0}$.
\end{theorem}

Since a module is finite dimensional if and only if its support contains no line through the origin, we obtain a necessary condition for finite dimensionality. Checking this condition can be done using the characteristic zero representation theory of the finite group $G$. 
Observing that $G_S$ is a subgroup of $N_S$ we obtain the following corollary, which is easier to check in practice.

\begin{corollary} \label{simplemain}
If the module $L_c(\lambda)$ is finite dimensional, then for each $S = \CC^\times v$ as in Theorem \ref{main}, the restriction $\Res^G_{G_S} S^\lambda$ belongs to the Serre subcategory of $\CC G_S$-mod generated by the modules $\Res^G_{G_S} S^\mu$, for those $\mu$ such that $c_\mu - c_\lambda \in \ZZ_{> 0}$.
\end{corollary}

This corollary may be proved independently of the theorem; since it requires very little machinery we give the proof in the next subsection, so that this introduction constitutes a self-contained explanation of the fact that $L_c(\mathrm{triv})$ is the only possible finite dimensional simple $H_c(S_n,V)$-module when $c>0$.

\subsection{Proof of Corollary \ref{simplemain}} For a point $p \in V$ and a module $M$ for $H_c(G,V)$, we let $I(p)$ be the ideal of $p$ in $\CC[V]$ and write $M(p)=\CC[V] / I(p) \otimes_{\CC[V]} M$ for the fiber of $M$ at $p$, viewing $M$ as a $\CC[V]$-module only. Evidently $M(p)$ remembers very little of the structure of $M$, but it is at least a $G_p$-module. For example, for a standard module $\Delta_c(\eta)$, we get $\Delta_c(\eta)(p) = \CC[V] / I(p) \otimes_{\CC[V]} (\CC[V] \otimes S^\eta) = \Res^G_{G_p} S^\eta$.

The functor $M \mapsto M(p)$ is a right-exact functor from $\OO_c$ to the category $\CC G_p$-mod of finitely generated modules for $\CC G_p$.  If $0=M_0 \subseteq M_1 \subseteq \cdots \subseteq M_k=M$ is a filtration by submodules, then the fiber $M(p)$ is filtered by $\CC G_p$-submodules $$0=M_0(p) \subseteq \overline{M_1(p)} \subseteq \overline{M_2(p)} \subseteq \cdots \subseteq \overline{M_k(p)}=M(p)$$ where we have written $\overline{M_i(p)}$ for the image of $M_i(p)$ in $M(p)$ by the canonical map. By right exactness, for a quotient $M'/M''$ of $\CC[V]$-modules, the fiber $(M'/M'')(p)$ is isomorphic to the quotient of $M'(p)$ by the image of $M''(p)$ in $M'(p)$. Thus the layers of this filtration are quotients of the fibers $(M_{i+1}/M_i)(p)$ of the layers of the original filtration.

From the short exact sequence
$$0 \longrightarrow \mathrm{Rad}(\Delta_c(\lambda)) \longrightarrow \Delta_c(\lambda) \longrightarrow L_c(\lambda) \longrightarrow 0$$
we get the following exact sequence by applying the right exact functor $(-)(p)$, for any $p \neq 0$:
$$\mathrm{Rad}(\Delta_c(\lambda))(p) \longrightarrow \Res^G_{G_p} S^\lambda \longrightarrow 0$$
since the finite dimensional module $L_c(\lambda)$ is supported at $\{0\}$.

The module $\mathrm{Rad}(\Delta_c(\lambda))$ has a filtration with layers isomorphic to various $L_c(\mu)$'s, all with $\mu >_c \lambda$. The fiber $L_c(\mu)(p)$ is a quotient of
$\Delta_c(\mu)(p) = \Res^G_{G_p}S^\mu$. By the previous paragraph, it follows that $\mathrm{Rad}(\Delta_c(\lambda))(p)$ has a filtration with layers that are isomorphic to quotients of various $\Res^G_{G_p} S^\mu$'s, all with $\mu >_c \lambda$. For $p = v$, we have $G_p = G_S$. This finishes the proof.

\subsection{Finite dimensionals for symmetric group} We explain how to use the corollary in the case of the symmetric group $S_n$, where the classification of finite dimensional modules was already known. There is only one conjugacy class of reflections, the transpositions $(ij)$. We assume the corresponding number $c=c_r$ is positive and real (for nonreal $c$ it's easy to see that there are no finite dimensional modules, and the case of negative real $c$ is easily reduced to positive real $c$). The irreducible representations of $S_n$ are indexed by partitions $\lambda$ of $n$, and the $c$-function is given by
$$c_\lambda=c \left(\frac{n(n-1)}{2}- \sum_{b \in \lambda} \mathrm{ct}(b)\right),$$ where $\mathrm{ct}(b)$ denotes the content of the box $b$. Using only the corollary and classical facts about representations of $S_n$, we will prove:
\begin{corollary}
In the case of the symmetric group $S_n$ acting on its reflection representation $V = \{v \in \CC^n \ | \ v_1 + v_2 + \cdots + v_n=0 \}$, if $c>0$ and $L_c(\lambda)$ is finite dimensional then $\lambda=(n)$ is a single row and $nc \in \ZZ_{>0}$.
\end{corollary}
\begin{proof} The maximal parabolic subgroups of $S_n$ are, up to conjugacy, all of the form $S_k \times S_{n-k}$; the corresponding stratum is the set of points $(u, u, \dots, u, v, v, \dots, v)$ with $k$ equal $u$'s and $n-k$ equal $v$'s with $u \neq v$. The restriction of $S^\lambda$ to $S_k \times S_{n-k}$ is given by
$$\mathrm{res}_{S_k \times S_{n-k}}^{S_n} S^\lambda=\bigoplus (S^\mu \otimes S^\nu)^{\oplus c_{\mu, \nu}^\lambda}$$ where $\mu$ runs over partitions of $k$, $\nu$ runs over partitions of $n-k$, and $c_{\mu,\nu}^\lambda$ are the Littlewood-Richardson numbers. The Pieri rule computes $c^\lambda_{\mu,\nu}$ in the special case when one of $\mu$ or $\nu$ is a single row: if $\mu = (k)$ is a single row of length $k$, then $c^\lambda_{(k),\nu}$ is zero unless $\lambda$ is obtained from $\nu$ by adding $k$ boxes, no two in the same column. In particular, if $\lambda$ is a partition other than $(n)$ with last row of length $k$ (so $k<n$), $\mu$ is a single row of length $k$ and $\nu$ is the partition of $n-k$ obtained from $\lambda$ by removing its last row, then the Pieri rule implies that $\lambda$ is maximal in $c$-order among those partitions with $c^\lambda_{(k),\nu} \neq 0$: any other placement of $k$ boxes along the border of $\nu$, no two in the same column, increases the sum of the contents. In other words, $S^\lambda$ contains $S^{(k)} \otimes S^\nu$ in its restriction to $S_k \times S_{n-k}$ and is maximal in the $c$-order with this property. Using Corollary~\ref{simplemain} this implies that $L_c(\lambda)$ is not finite dimensional.

Here we have used only that $c$ is positive. Analyzing what can happen for the trivial module gives the condition $nc \in \ZZ$ as follows: consider the restriction with  $G_S=S_{n-1} \times S_1=S_{n-1}$. The only partitions $\lambda$ such that $\mathrm{res}^{S_n}_{S_{n-1}} S^\lambda$ contains the trivial representation are $(n)$ and $(n-1,1)$. Then the condition that $(n) <_c (n-1,1)$ is precisely $nc \in \ZZ_{>0}$.
\end{proof} This corollary reduces the classification of finite dimensional modules for the $S_n$ Cherednik algebra to the problem of classifying the values of $c$ for which $L_c(\mathrm{triv})$ is finite dimensional; as of this writing we know of three essentially different ways to do this: Etingof's analysis of the relationship with the Macdonald-Mehta integral, Varagnolo-Vasserot's approach through affine flag varieties, and Cherednik's technique of intertwining operators.

\section{Cherednik algebras} \label{basics}

This section reviews standard facts about Cherednik algebras. 

\subsection{Definition of the Cherednik algebra} Suppose $V$ is a finite dimensional $\CC$-vector space and $G \subseteq \mathrm{GL}(V)$ is a finite group of linear transformations of $V$. A \emph{reflection} is an element $r \in G$ whose fixed space is of codimension exactly $1$. A \emph{reflecting hyperplane} for $G$ is the fixed space of a reflection $r \in G$. We write $R$ for the set of reflections in $G$, $\mathcal{A}$ for the set of reflecting hyperplanes of $G$ on $V$, and for each element $H \in \mathcal{A}$, let $\alpha_H \in V^*$ be a linear form with kernel $H$.

Fix a $G$-invariant function on the set of reflections in $G$, $r \mapsto c_r$, and for each reflection $s$ we write $\alpha_r=\alpha_H$ if $H$ is the fixed space of $s$. The \emph{rational Cherednik algebra} for these data is the subalgebra $H_c(G,V)$ of $\mathrm{End}_\CC(\CC[V])$ generated by $\CC[V]$, the group $G$, and for each $y \in V$ a \emph{Dunkl operator}
$$y(f)=\partial_y(f)-\sum_{r \in R} c_r \la \alpha_r,y \ra \frac{f-r(f)}{\alpha_r} \quad \hbox{for $f \in \CC[V]$.}$$ The Dunkl operators commute with one another, and the \emph{Poincar\'e-Birkhoff-Witt theorem} for $H_c(G,V)$ asserts that multiplication induces an isomorphism of vector spaces
$$\CC[V] \otimes_\CC \CC G \otimes_\CC \CC[V^*] \stackrel{\cong}{\longrightarrow} H_c(G,V).$$

The algebra $H_c(G,V)$ is graded with $V^*$ in degree $1$, the group $G$ in degree $0$, and the space $V$ of Dunkl operators in degree $-1$. This grading is actually internal: the Euler element 
\[
\mathrm{eu} = \sum_{1\leq j \leq n} x_j \frac{\partial}{\partial x_j}
			= \sum_{1\leq j \leq n} x_j y_j + \sum_{r \in R} c_r (1-r)
			\in H_c(G,V),
\]
where $(x_1, \dots, x_n)$ and $(y_1, \dots, y_n)$ are dual bases of $V^*$ and $V$, satisfies
\[
[\mathrm{eu}, x] = x,\ [\mathrm{eu}, y] = -y,\text{ and }[\mathrm{eu}, g]=0
\qquad (\forall x \in V^*,\ \forall y \in V,\ \forall g \in G).
\]

\cut{
 arising from the adjoint action of the \emph{Euler element} $\mathrm{eu}=\sum_{1\leq j \leq n} x_j \frac{\partial}{\partial x_j}$, where $(x_1, \dots, x_n)$ is an arbitrary basis of $V$. This element lies in $H_c(G,V)$ and may be written in terms of the generators as
\begin{equation}
 \mathrm{eu} =
  \sum x_i y_i+\sum_{r \in R} c_r (1-r),
  \quad \text{with} 
  \quad [\mathrm{eu},x]=x, \ [\mathrm{eu},y]=-y, \ \text{and} \ [\mathrm{eu},g]=0
\end{equation}
for all $x \in V^*$, $y \in V$, and $g \in G$, where $(y_1, \dots, y_n)$ is the basis of $V^*$ dual to $(x_1,\dots,x_n)$.
}

\subsection{Standard modules} The subalgebra of $H_c(G,V)$ generated by $G$ and the Dunkl operators is isomorphic to the twisted group ring $\CC[V^*] \rtimes G$, and we inflate each $\CC G$-module $M$ to a $\CC[V^*] \rtimes G$-module, so that the Dunkl operators act by $0$ on $M$. Fix an index set $\Lambda$ for the isomorphism classes of irreducible $\CC G$-modules, and for $\lambda \in \Lambda$ write $S^\lambda$ for the corresponding irreducible module. The \emph{standard modules} for the Cherednik algebra are given by
$$\Delta_c(\lambda)=\mathrm{Ind}_{\CC[V^*] \rtimes G}^{H_c(G,V)} S^\lambda.$$ In fact, the PBW theorem implies that there is an isomorphism  of $\CC$-vector spaces $\Delta_c(\lambda) \cong \CC[V] \otimes_\CC S^\lambda$, and the $H_c(G,V)$-module structure is given by $\CC[V] \rtimes G$ acting in the obvious way, and with the Dunkl operators acting by
$$y(f \otimes u)=\partial_y(f) \otimes u - \sum_{r \in R} c_r \la \alpha_r,y \ra \frac{f-r(f)}{\alpha_r} \otimes r(u) \quad \hbox{for all $y \in V$, $f \in \CC[V]$, and $u \in S^\lambda$.}$$

Category $\mathcal{O}_c(G,V)$ is the full subcategory of $H_c(G,V)$-mod consisting of modules that are locally nilpotent for the action of the Dunkl operators and finitely generated over the polynomial subalgebra $\CC[V]$. As such each module in $\mathcal{O}_c(G,V)$ gives rise to a $G$-equivariant coherent sheaf on $V$. The  Euler element $\mathrm{eu}$ acts locally finitely on each module in $\OO_c(G,V)$. It follows that each module $M \in \OO_c(G,V)$ is a graded $H_c(G,V)$-module, with grading defined by the finite dimensional subspaces
$$M^d=\{m \in M \ | \ \hbox{there is $N \in \ZZ_{>0}$ with $(\mathrm{eu}-d)^N m=0$} \}.$$ Evidently the morphisms in $\OO_c(G,V)$ are all degree $0$ with respect to this grading. The grading on the standard modules is given by
$$\Delta_c(\lambda)=\bigoplus_{d \in \ZZ_{\geq 0}} \CC[V]^d \otimes S^\lambda,$$ with $\CC[V]^d$ the space of polynomials of degree $d$ and $\CC[V]^d \otimes S^\lambda$ in degree $d+c_\lambda$ where $c_\lambda$ is the scalar by which $\sum c_r (1-r)$ acts on $S^\lambda$.

Given a complex number $z \in \CC$, we write $\OO_c^{> z}$ for the full subcategory consisting of objects $M$ of $\OO_c$ such that $M^d=0$ unless $d-z \in \ZZ_{>0}$.

The standard module $\Delta_c(\lambda)$ has a unique irreducible quotient $L_c(\lambda)=\Delta_c(\lambda)/ \mathrm{Rad}(\Delta_c(\lambda))$, and the radical belongs to $\OO_c^{>c_\lambda}$,
\begin{equation} \label{radical bound}
\mathrm{Rad}(\Delta_c(\lambda)) \in \OO_c^{>c_\lambda}.
\end{equation} If $M$ is a finite dimensional irreducible $H_c(G,V)$-module, then $M \cong L_c(\lambda)$ for some $\lambda \in \Lambda$.
\section{Parabolic degeneration}

In this section we explain the degeneration procedure for Cherednik algebra modules we will use, emphasizing the monodromy representations that are produced. Our approach also produces an alternative construction of Bezrukavnikov-Etingof parabolic restriction; though we don't use this we explain it in \ref{BezEt}. Some of the results we obtain are parallel to those obtained by Bellamy-Ginzburg in section 8 of \cite{BeGi}; in particular, part (a) of our Theorem \ref{parabolic Dunkl formula} is closely related to their Theorem 8.4.5. The technical details of the proofs are different---in particular, we avoid the centralizer algebra construction in favor of a more geometric point of view, and we emphasize explicit formulas in order to calculate eigenvalues of monodromy. 

We have found the papers \cite{GoMa} and \cite{Wil} very inspirational. In fact, the constructions in this section become more natural when placed in the general context treated in the first half of Wilcox's paper: Etingof's sheaf of Cherednik algebras on a variety with a finite group action. We have avoided this generality and the additional notation it requires because it is not strictly necessary for obtaining our results. We have, however, kept our notation consistent with Etingof's: when $X$ is a variety (for us, always affine) with action of a finite group $G$, the notation $H_c(G,X)$ always means the Cherednik algebra for $G$ acting on $X$.

\subsection{Stratification} Define an equivalence relation on $V$ by $p \simeq q$ if $G_p=G_q$. The equivalence classes for this equivalence relation are called \emph{strata}. Given a stratum $S$, its closure $\overline{S}$ consists of points whose stabilizer group contains the stabilizer group of a generic point.

Suppose $G$ is generated by $R$. By Theorem 1.5 of \cite{Ste}, each group $G_p$ is generated by $R \cap G_p$. It follows that each stratum $S$ is a hyperplane complement in its closure, $$S=\overline{S} \setminus \bigcup_{\substack{H \in \mathcal{A} \\ S \not\subseteq H}} H.$$ Algebraically, $S$ is the locally closed subvariety of $V$ defined by
$$S=\{v \in V \ | \ \alpha_H(v)=0 \quad \hbox{for all $H \in \mathcal{A}$ with $S \subseteq H$ and} \prod_{\substack{H \in \mathcal{A} \\ S \not\subseteq H}} \alpha_H(v) \neq 0 \}$$

\subsection{The \'etale pullback associated to a subset} Let $S \subseteq V$ be a subset, let
$$G_S= \bigcap_{p \in S} G_p \subseteq G$$ be the pointwise stabilizer of $S$ and let $$N_S=\{g \in G \ | \ g(S) \subseteq S \}$$ be the setwise stabilizer of $S$. Let $$R_S=R \cap G_S$$ be the set of reflections in $G_S$ and let $\mathcal{A}_S$ be the set of reflecting hyperplanes for elements of $R_S$ (these are precisely the hyperplanes in $\mathcal{A}$ containing the set $S$). For each coset $g N_S$ of $N_S$ in $G$ (abbreviated by $g=gN_S$ when confusion will not result), let $$V_g=V \setminus \bigcup_{H \notin g(\mathcal{A}_S)} H$$ be the complement in $V$ of the set of reflecting hyperplanes for elements not in $g G_Sg^{-1}$. Finally, let $$X_S=\coprod_{g \in G/N_S} V_g \subseteq G/N_S \times V$$ be the disjoint union of the spaces $V_g$. The group $G$ acts on $G/N_S \times V$ diagonally by the formula $g(hN_S,v)=(ghN_S,gv)$, and the (open) subvariety $X_S$ is $G$-stable. The projection on the second factor gives a $G$-equivariant \'etale map $p$ from $X_S$ into $V$, corresponding to the diagonal embedding
$$p^*:\CC[V] \rightarrow \CC[X_S]=\bigoplus_{g \in G/N_S} \CC[V_g].$$

Writing $V^\circ$ and $X_S^\circ$ for the sets of points with trivial stabilizer in $G$, we have $p(X_S^\circ)=V^\circ$: if $v \in H$ for some $H \in g(\mathcal{A_S})$ then the point $(g,v)$ is stabilized by any reflection with fixed space $H$.

\subsection{Differential operators} We write $D(V^\circ)$ for the ring of differential operators on $V^\circ$ and similarly write $D(X_S^\circ)$ for the ring of differential operators on $X_S^\circ$. The \'etale map $p:X_S^\circ \rightarrow V^\circ$ induces a homomorphism
$$p^*:D(V^\circ) \rtimes G \rightarrow D(X_S^\circ) \rtimes G,$$ given on generators by the formulas
$$\partial_y \mapsto \partial_y \quad \text{and} \quad f \mapsto f \circ p,$$ where we have written $\partial_y$ for the differential operator on $X_S$ corresponding to the constant vector field associated to $y \in V$. Evidently $H_c(G,V)$ may be regarded as the subalgebra of $D(V^\circ) \rtimes G$ generated by Dunkl operators, the group $G$, and the ring $\CC[V]$ of functions on $V$.

\subsection{The rational Cherednik algebra for $X_S$} The rational Cherednik algebra for $G$ acting on $X_S$ is the subalgebra $H_c(G,X_S)$ of $D(X_S^\circ) \rtimes G$ generated by $\CC[X_S]$, the group $G$, and for each $y \in V$ a Dunkl operator given by the same formula as before:
$$y(f)=\partial_y(f)-\sum_{r \in R} c_r \la \alpha_r,y \ra \frac{f-r(f)}{\alpha_r} \quad \hbox{for $f \in \CC[X_S]$.}$$ Here we have interpreted $\alpha_r$ as an element of $\CC[X_S]$ via the diagonal embedding as before; as such it is a generator for the ideal of functions vanishing on the fixed point subvariety of $s$ acting on $X_S$ (this subvariety may be empty, in which case $\alpha_r$ is a unit and our claim is still true). The Cherednik algebra $H_c(G,V)$ sits inside $H_c(G,X_S)$ as the subalgebra generated by Dunkl operators and the diagonally embedded copy of $\CC[V]$. The PBW theorem for $H_c(G,X_S)$ is
\begin{theorem}
The Dunkl operators commute, and multiplication induces isomorphisms of $\CC[X_S]$-modules
$$\CC[X_S] \otimes_\CC \CC G \otimes_\CC \CC[V] \cong H_c(G,X_S) \cong \CC[X_S] \otimes_{\CC[V]} H_c(G,V).$$  
\end{theorem}
\begin{proof}
This is an easy consequence of the analogous facts for the rational Cherednik algebra $H_c(G,V)$. The first part of this theorem, that the Dunkl operators commute, is a consequence of the fact that the usual Dunkl operators commute: the Dunkl operators for $G$ acting on $X_S$ are the images under the homomorphism $D(V^\circ) \rtimes G \rightarrow D(X_S^\circ) \rtimes G$ of the usual Dunkl operators. Identifying $H_c(G,V)$ with a subalgebra of $D(X_S^\circ) \rtimes G$ via the composition
$$H_c(G,V) \hookrightarrow D(V^\circ) \rtimes G \hookrightarrow D(X_S^\circ) \rtimes G $$ it follows that $H_c(G,X_S)$ is the subalgebra of $D(X_S^\circ) \rtimes G$ generated by $H_c(G,V)$ and $\CC[X_S]$. From the commutation formula
$$yf-fy=\partial_y(f)-\sum_{r \in R} c_r \la \alpha_r,y \ra \frac{f-r(f)}{\alpha_r} r$$ we see that multiplication $\CC[X_S] \otimes_{\CC[V]} H_c(G,V) \rightarrow H_c(G,X_S)$ is surjective. It is injective because it is the restriction of the isomorphism $$\CC[X_S^\circ] \otimes_{\CC[V]} (D(V^\circ) \rtimes G)=\CC[X_S^\circ] \otimes_{\CC[V^\circ]} (D(V^\circ) \rtimes G) \rightarrow D(X_S^\circ) \rtimes G.$$ \end{proof}

\subsection{Localization of $H_c(G,V)$-modules} For an $H_c(G,V)$-module $M$ we have
$$\mathrm{Ind}_{H_c(G,V)}^{H_c(G,X_S)}(M)=\CC[X_S] \otimes_{\CC[V]} M.$$ Since $\CC[X_S]$ is a flat $\CC[V]$-module, it follows that induction defines an exact functor from $H_c(G,V)$-mod to $H_c(G,X_S)$-mod.

For each $g \in G/N_S$, let $e_g \in \CC[X_S]$ be the characteristic function of $V_g \subseteq X_S$ and abbreviate $e=e_1$. The equations $1=\sum_{g \in G/N_S} e_g$ and $h e_g h^{-1}=e_{hg}$ for $g \in G/N_S$ and $h \in G$ imply that $H_c(G,X_S) e H_c(G,X_S)=H_c(G,X_{S})$, and hence by Morita theory the functor $M \mapsto eM$ defines an equivalence $H_c(G,X_S) \mathrm{-mod} \cong e H_c(G,X_S) e \mathrm{-mod}$. Writing $V_S=V \setminus \bigcup_{H \notin \mathcal{A_S}} H$ it follows that the composition of this idempotent slice functor with induction is given by
$$e \mathrm{Ind}_{H_c(G,V)}^{H_c(G,X_S)}(M)=\CC[V_S] \otimes_{\CC[V]} M.$$

\subsection{The Cherednik algebra $H_c(N_S,V_S)$} Recall that according to our notation, $V_S=V \setminus \bigcup_{H \notin \mathcal{A_S}} H$ is the complement in $V$ of the set of reflecting hyperplanes for reflections in $R \setminus R_S$. The Cherednik algebra $H_c(N_S,V_S)$ of $N_S$ acting on $V_S$ is the subalgebra of $\mathrm{End}_\CC(\CC[V_S])$ generated by $\CC[V_S]$, $N_S$, and for each $y \in V$, a Dunkl operator $y_S$ defined by
$$y_S(f)=\partial_y f-\sum_{r \in R_S} c_r \la \alpha_r,y \ra \frac{f-r(f)}{\alpha_r} \quad \hbox{for all $f \in \CC[V_S]$.}$$

\begin{theorem} \label{parabolic Dunkl formula}
\item[(a)] There is an equality $H_c(N_S,V_S) = eH_c(G,X_S)e$ (both are subalgebras of $\mathrm{End}_\CC(\CC[V_S])$).
\item[(b)] For $y \in V$, we have the formula
$$y_S=eye-\sum_{r \in (R \setminus R_S) \cap N_S} c_r \la \alpha_r,y \ra \frac{1}{\alpha_r} r e+\sum_{r \in R \setminus R_S} c_r \la \alpha_r,y \ra \frac{1}{\alpha_r}e.$$
\end{theorem}
\begin{proof} We prove (b) first.
Calculate, for $f \in \CC[V_S]$,
\begin{align*}
&\left( eye-\sum_{r \in  (R \setminus R_S) \cap N_S} c_r \la \alpha_r,y \ra \frac{1}{\alpha_r} r e \right) (f)=ey(f)-\sum_{r \in  (R \setminus R_S) \cap N_S} c_r \la \alpha_r,y \ra \frac{1}{\alpha_r} r(f) \\
&=e\partial_y(f)-e\sum_{r \in R} c_r \la \alpha_r,y \ra \frac{f-r(f)}{\alpha_r}-\sum_{r \in  (R \setminus R_S) \cap N_S} c_r \la \alpha_r,y \ra \frac{1}{\alpha_r} r(f) \\
&=\partial_y(f)-e\sum_{r \in R \cap N_S} c_r \la \alpha_r,y \ra \frac{f-r(f)}{\alpha_r}-e\sum_{r \in R \setminus N_S} c_r \la \alpha_r,y \ra \frac{f-r(f)}{\alpha_r}-\sum_{r \in  (R \setminus R_S) \cap N_S} c_r \la \alpha_r,y \ra \frac{1}{\alpha_r} r(f) \\
&=\partial_y(f)-\sum_{r \in R \cap N_S} c_r \la \alpha_r,y \ra \frac{f-r(f)}{\alpha_r}- \sum_{r \in R \setminus N_S} c_r \la \alpha_r,y \ra \frac{f}{\alpha_r}-\sum_{r \in  (R \setminus R_S) \cap N_S} c_r \la \alpha_r,y \ra \frac{1}{\alpha_r} r(f) \\
&=\partial_y(f)-\sum_{r \in  (R \setminus R_S) \cap N_S} c_r \la \alpha_r,y \ra \frac{f-r(f)}{\alpha_r} -\sum_{r \in R_S} c_r \la \alpha_r,y \ra \frac{f-r(f)}{\alpha_r} \\
&-\sum_{r \in R \setminus N_S} c_r \la \alpha_r,y \ra \frac{f}{\alpha_r}-\sum_{r \in  (R \setminus R_S) \cap N_S} c_r \la \alpha_r,y \ra \frac{1}{\alpha_r}r(f) \\
&=\partial_y(f)-\sum_{r \in  (R \setminus R_S) \cap N_S} c_r \la \alpha_r,y \ra \frac{f}{\alpha_r}-\sum_{r \in R_S} c_r \la \alpha_r,y \ra \frac{f-r(f)}{\alpha_r}-\sum_{r \in R \setminus N_S} c_r \la \alpha_r,y \ra \frac{f}{\alpha_r} \\
&=\partial_y(f)-\sum_{r \in R_S} c_r \la \alpha_r,y \ra \frac{f-r(f)}{\alpha_r}-\sum_{r \in R \setminus R_S} c_r \la \alpha_r,y \ra \frac{f}{\alpha_r}=y_S(f)-\sum_{r \in R \setminus R_S} c_r \la \alpha_r,y \ra \frac{f}{\alpha_r}.
\end{align*}
This proves (b), and the containment $H_c(N_S,V_S) \subseteq eH_c(G,X_S)e$.

To prove the reverse containment, and hence part (a), we proceed by induction on the polynomial degree of $\varphi$ to prove that any expression of the form $e f \varphi g e$, with $f \in \CC[X_S]$, $\varphi \in \CC[V^*]$, and $g \in G$, is contained in $H_c(N_S,V_S)$. If the degree of $\varphi$ is zero, then since $ef = fe$ and $e g e=0$ unless $g \in N_S$, we have $e f g e \in H_c(N_S,V_S)$. Hence to prove the general case it is enough to prove that $e f y \varphi g e \in H_c(N_S,V_S)$ for $y\in V$ and $\varphi\in\CC[V^*]$. Using $e^2=e$ and $ef = fe$, we compute
\begin{align*}
e f y \varphi g e&=e f e^2 y \varphi g e=e f e \left([e,y]+ye \right)\varphi g e \\
&=e f e [e,y] \varphi g e+e fe y e \varphi g e.
\end{align*}
Since $[e,y] \in \CC[X_S] \otimes \CC G$, the induction hypothesis implies that both of these last terms are in $H_c(N_S,V_S)$, and this completes the proof.
\end{proof}
The next lemma explicitly computes the action of the Dunkl operators $y_S$ on the restrictions of the standard modules. The formula is crucial for the proof of Theorem \ref{monodromy eigenvalues}.
\begin{lemma} \label{stratum Dunkl} The action of the Dunkl operators $y_S$ on the localization of the standard module $\Delta_c(\lambda)|_{V_S}=\CC[V_S] \otimes_\CC S^\lambda$ is by
$$y_S(f \otimes u)=\partial_y(f) \otimes u-\sum_{r \in R_S}c_r \la \alpha_r,y \ra \frac{f-r(f)}{\alpha_r} \otimes r(u)-\sum_{r \in R \setminus R_S}c_r \la \alpha_r,y \ra \frac{f}{\alpha_r} \otimes (r(u)-u)).$$
\end{lemma}
\begin{proof} Compute
\[
\begin{array}{rcl}
y_S(f \otimes u)
&=&\displaystyle
\left(eye-\sum_{r \in  (R \setminus R_S) \cap N_S} c_r \la \alpha_r,y \ra \frac{1}{\alpha_r}re+\sum_{r \in R \setminus R_S} c_r \la \alpha_r,y \ra \frac{1}{\alpha_r}e\right)(f \otimes u) \\
\\
&=&\displaystyle
ey(f \otimes u)-\sum_{r \in  (R \setminus R_S) \cap N_S} c_r \la \alpha_r,y \ra \frac{r(f)}{\alpha_r} \otimes r(u)+\sum_{r \in R \setminus R_S}c_r \la \alpha_r,y \ra \frac{f}{\alpha_r} \otimes u \\
&=&\displaystyle
\partial_y(f) \otimes u-e\sum_{r \in R} c_r \la \alpha_r,y \ra \frac{f-r(f)}{\alpha_r} \otimes r(u)-\sum_{r \in  (R \setminus R_S) \cap N_S}c_r \la \alpha_r,y \ra \frac{r(f)}{\alpha_r} \otimes r(u) \\
&&\displaystyle
+\sum_{r \in R \setminus R_S}c_r \la \alpha_r,y \ra \frac{f}{\alpha_r} \otimes u \\
&=&\displaystyle
\partial_y(f) \otimes u- \sum_{r \in R \cap N_S}c_r \la \alpha_r,y \ra \frac{f-r(f)}{\alpha_r} \otimes r(u)-\sum_{r \in R \setminus N_S}c_r \la \alpha_r,y \ra \frac{f}{\alpha_r} \otimes r(u) \\
&&\displaystyle
-\sum_{r \in  (R \setminus R_S) \cap N_S}c_r \la \alpha_r,y \ra \frac{r(f)}{\alpha_r} \otimes r(u)+\sum_{r \in R \setminus R_S}c_r \la \alpha_r,y \ra \frac{f}{\alpha_r} \otimes u \\
&=&\displaystyle
\partial_y(f) \otimes u-\sum_{r \in R_S}c_r \la \alpha_r,y \ra \frac{f-r(f)}{\alpha_r} \otimes r(u)-\sum_{r \in  (R \setminus R_S) \cap N_S}c_r \la \alpha_r,y \ra \frac{f}{\alpha_r} \otimes r(u) \\
&&\displaystyle
-\sum_{r \in R \setminus N_S}c_r \la \alpha_r,y \ra \frac{f}{\alpha_r} \otimes r(u)+\sum_{r \in R \setminus R_S}c_r \la \alpha_r,y \ra \frac{f}{\alpha_r} \otimes u \\
&=&\displaystyle
\partial_y(f) \otimes u-\sum_{r \in R_S}c_r \la \alpha_r,y \ra \frac{f-r(f)}{\alpha_r} \otimes r(u)-\sum_{r \in R \setminus R_S}c_r \la \alpha_r,y \ra \frac{f}{\alpha_r} \otimes r(u) \\
&&\displaystyle
+\sum_{r \in R \setminus R_S}c_r \la \alpha_r,y \ra \frac{f}{\alpha_r} \otimes u \\
&=&\displaystyle
\partial_y(f) \otimes u-\sum_{r \in R_S}c_r \la \alpha_r,y \ra \frac{f-r(f)}{\alpha_r} \otimes r(u)-\sum_{r \in R \setminus R_S}c_r \la \alpha_r,y \ra \frac{f}{\alpha_r} \otimes (r(u)-u)).\end{array}
\]
\end{proof}

\subsection{Degeneration to the normal bundle}

Let $I_S$ be the ideal of $S$ in $\CC[V_S]$ where as above $V_S=V \setminus \bigcup_{H \notin \mathcal{A}_S} H$. Given a $\CC[V_S]$-module $M$ we define the associated graded module for the $I_S$-adic filtration by
$$\mathrm{gr}_{I_S}(M)=\bigoplus_{k \in \ZZ_{\geq 0}} I_S^k M/ I_S^{k+1} M.$$ This is the \emph{degeneration of $M$ to the normal bundle of $S$ in $V_S$}.

We will assume from now on that $M$ is finitely generated as a $\CC[V_S]$-module (for us, the most important case is when $M$ is the restriction to $V_S$ of a module in $\OO_c(G,V)$). We write $S^\perp$ for the vector subspace of $V$ generated by the images of $1-r$ for all $r \in R_S$; this is also the orthogonal complement of $\overline S$ with respect to any $G_S$-invariant positive definite Hermitian form on $V$, or the sum of the non-trivial $G_S$-submodules of $V$.
\begin{lemma}
\begin{enumerate}
\item[(a)] If $M$ is a finitely generated as a $\CC[V_S]$-module, then $\mathrm{gr}_{I_S}(M)$ is finitely generated as a $\mathrm{gr}_{I_S}(\CC[V_S])$-module, and each graded component $\mathrm{gr}^j_{I_S}(M)$ is a finitely generated $\mathrm{gr}^0_{I_S}(\CC[V_S])=\CC[S]$-module.
\item[(b)] Let $U$ be the span of the set $\alpha_r$ for $r \in R_S$. The canonical maps $$\CC[S] \otimes_\CC \mathrm{Sym}^d U \rightarrow  \mathrm{gr}^d_{I_S} \CC[V_S]$$ are vector space isomorphisms, and fit together to give an $N_S$-equivariant algebra isomorphism
$$\CC[S \times S^\perp]=\CC[S] \otimes_\CC \CC[S^\perp] \rightarrow \mathrm{gr}_{I_S} \CC[V_S].$$
\end{enumerate}
\end{lemma}
\begin{proof}
The fact that $\mathrm{gr}_{I_S}(M)$ is finitely generated is Proposition 5.2 from \cite{Eis}; the fact that its graded components are finitely generated $\CC[S]$-modules is a consequence of this, proving (a). Part (b) is proved if we observe that 
$$I_S^d=\CC[V_S] U^d=\CC[S] U^d \quad \hbox{mod $I_S^{d+1}$.}$$
\end{proof}

\subsection{Parabolic degeneration of $H_c(N_S,V_S)$} Filter $H_c(N_S,V_S)$ by subspaces $F^k$, for $k \in \ZZ$ running over all integers, given by the formula
\[
F^k = \{a \in H_c(N_S,V_S) \ | \ a I_S^j H_c(N_S,V_S) \subseteq I_S^{j+k} H_c(N_S,V_S) \quad \hbox{for all $j \in \ZZ_{\geq 0}$} \}.
\]
Then $F^{k+1} \subseteq F^k$ and $F^j F^k \subseteq F^{j+k}$, so we define the associated graded algebra 
\[
\mathrm{gr}_F(H_c(N_S,V_S))=\bigoplus_{k \in \ZZ} F^k / F^{k+1}.
\]
 It acts on $\mathrm{gr}_{I_S}(M)$ for each $H_c(N_S,V_S)$-module $M$, and evidently a morphism $M \rightarrow N$ of $H_c(N_S,V_S)$-modules gives rise to a morphism $\mathrm{gr}_{I_S} (M) \rightarrow \mathrm{gr}_{I_S} (N)$ of the corresponding $\mathrm{gr}_F(H_c(N_S,V_S))$-modules. We have thus a right-exact functor $M \mapsto \mathrm{gr}_{I_S}(M)$. 

Recall the notation $S^\perp$ and $U$ from last subsection, and let $U' = (V^*)^{G_S}$, so that
\[
V = \overline{S} \oplus S^\perp \quad \text{and} \quad V^* = U\oplus U'.
\]
We will use the following abuses of notation to denote certain elements of $\mathrm{gr}_F(H_c(N_S,V_S))$:
\begin{enumerate}
\item[(a)] For each $r \in R_S$, we will again write $\alpha_r \in F^1/F^2$ for its image,
\item[(b)] for each $y \in S^\perp$, we will again write $y \in F^{-1}/F^0$ for its image, 
\item[(c)] for each $g \in N_S$ we will again write $g \in F^0 / F^1$ for its image,
\item[(d)] for each $y \in \overline{S}$ we will again write $y \in F^0/F^1$ for its image,
\item[(e)] for each $x \in U'$ and each $\alpha_r^{-1}$ with $r \notin R_S$ we will again write $x,\alpha_r^{-1} \in F^0 / F^1$ for the images modulo $F^1$.
\end{enumerate}
The operators from (a), (b), and (c) define a homomorphism
\[
H_c(G_S,S^\perp) \rightarrow \mathrm{gr}_F(H_c(N_S,V_S)),
\]
while those from (c), (d), and (e) define a homomorphism 
\[
D(S) \rtimes N_S \rightarrow \mathrm{gr}_F(H_c(N_S,V_S)).
\] 

We will not use it in the sequel, but we note that we have the following version of the PBW theorem for $\mathrm{gr}_F(H_c(N_S,V_S))$:
\begin{theorem}
Multiplication induces a vector space isomorphism
$$\CC[S^\perp] \otimes_\CC (D(S) \rtimes N_S) \otimes_\CC \mathrm{Sym}(S^\perp) \cong \mathrm{gr}_F(H_c(N_S,V_S)).$$
\end{theorem} We observe that the theorem may equivalently be stated: multiplication induces an isomorphism
$$\CC[S^\perp \times S] \otimes_\CC \CC N_S \otimes_\CC \CC[V^*] \cong \mathrm{gr}_F(H_c(N_S,V_S)).$$

\subsection{Parabolic degeneration of modules in $\OO_c$} For $M \in \OO_c(G,V)$, the $\CC[S]$-modules $\mathrm{gr^j}_{I_S}(M)$ are finitely generated for all $j$, and each is a $D(S) \rtimes N_S$-module: thus these are vector bundles with $N_S$-equivariant connections. For a point $p \in S$ and a finitely generated $D(S) \rtimes N_S$-module $A$ (the space of sections of a vector bundle that we will also denote by $A$) we write $\mathrm{Sol}_p(A)$ for the space of germs of holomorphic sections $f$ of $A$ near $p$ satisfying $yf=0$ for all $y \in \overline{S}$. Then we define the \emph{parabolic degeneration} of $M$ by
\begin{equation} \label{parabolic degeneration definition}
\mathrm{Deg}_{S,p}(M)=\bigoplus_{j=0}^\infty \mathrm{Sol}_p(\mathrm{gr}^j_{I_S}(M)).
\end{equation}
As $p$ varies this gives a (possibly infinite dimensional, but with finite dimensional graded pieces) local system on $S$. Each graded piece is a representation of the stratum braid group $B_S$, defined in the next section.

The space $\mathrm{Deg}_{S,p}(M)$ carries still more structure: for each $r \in R_S$, the operator $\alpha_r \in F^1/F^0$ commutes with each $y \in \overline{S}$. Likewise the operators $y' \in F^{-1}/F^0$ (for any $y' \in V$) all commute with $y \in F^0/F^1$ (again, for $y \in \overline{S}$), and the operators $g \in F^0/ F^1$ for $g \in N_S$ normalize the space of operators $y \in \overline{S}$. Thus all of these operators act on the solution spaces, and define a representation of $H_c(G_S,\overline{S}^\perp)$ on $\mathrm{Deg}_{S,p}(M)$. As such, it belongs to category $\OO_c$: it is non-negatively graded with finite dimensional graded pieces.

\subsection{Comparison with Bezrukavnikov-Etingof restriction}\label{BezEt}

We will not use the results of this subsection in the remainder of the paper; its purpose is only to indicate the relationship between what we have done and the original construction of Bezrukavnikov and Etingof.

Let $q \in V/G$ be a point and let $I=I(q)$ be the ideal of $q$ in $\CC[V/G]=\CC[V]^G$. Completing at $q$ defines a functor from $\CC[V/G]$-mod to $\widehat{\CC[V/G]}_{q}$-mod that we denote by $M \mapsto \widehat{M}_{q}$. By abuse of notation we will write $\widehat{V}_{q}$ for the formal scheme attached to $\widehat{\CC[V]}_{q}$. The group $G$ acts by automorphisms on this formal scheme, and the \emph{rational Cherednik algebra} $H_c(G,\widehat{V}_q)$ is the subalgebra of $\mathrm{End}_\CC(\widehat{\CC[V]}_q)$ generated by $\widehat{\CC[V]}_{q}$, the group $G$, and for each $y \in V$ a Dunkl operator
$$y(f)=\partial_y(f)-\sum_{r \in R} c_r \la \alpha_r,y \ra \frac{f-r(f)}{\alpha_r}.$$ There is an obvious embedding $H_c(G,V) \subseteq H_c(G,\widehat{V}_q)$ and via this embedding multiplication induces an isomorphism
$$\widehat{\CC[V]}_{q} \otimes_{\CC[V]} H_c(G,V) \cong H_c(G,\widehat{V}_q).$$

The algebra $\widehat{\CC[V]}_{q}$ is a direct sum over the points in the preimage of $q$, and fixing $p$ to be one of these we will write $e \in \widehat{\CC[V]}_{q}$ for the corresponding idempotent. Just as before we have $H_c(G,\widehat{V}_q)=H_c(G,\widehat{V}_q) e H_c(G,\widehat{V}_q)$ and so the idempotent slice functor $M \mapsto eM$ gives an equivalence from $H_c(G,\widehat{V}_q)$-mod to $e H_c(G,\widehat{V}_q) e$-mod. On the other hand, we have the following analog of Theorem \ref{parabolic Dunkl formula}: defining the Cherednik algebra for $G_p$ acting on $\widehat{V}_p$ as the subalgebra of $\End_\CC(\widehat{\CC[V]}_p)$ generated by $\widehat{\CC[V]}_p$, the group $G_p$, and for each $y \in V$ a \emph{Dunkl operator} 
$$y_p(f)=\partial_y(f)-\sum_{r \in R_p} c_r \la \alpha_r,y \ra \frac{f-r(f)}{\alpha_r},$$ where $R_p=G_p \cap R$ is the set of reflections fixing $p$, we have the equalities
$$e y e=y_p+\sum_{r \in R \setminus R_p} c_r \la \alpha_r,y \ra \frac{1}{\alpha_r} \quad \text{and hence} \quad e H_c(G,\widehat{V}_q) e= H_c(G_p,\widehat{V}_p).$$ The composite $M \mapsto e \widehat{M}_q$ thus defines an exact functor from $H_c(G,V)$-mod to $H_c(G_p,\widehat{V}_p)$-mod, which sends $M$ to $\widehat{M}_p$ as a $\CC[\widehat{V}_p] \rtimes G_p$-module. There are two further modifications necessary in order to obtain a module in category $\OO_c(G_p,V/V^{G_p})$. First, we consider the $H_c(G_p,V_p)$-submodule consisting of vectors $m$ that are killed by sufficiently high degree homogeneous polynomials in the Dunkl operators. Second, we consider the subspace of the resulting space consisting of vectors killed by the Dunkl operators $y \in V^{G_p}$. The appropriate analog of Theorem 2.3 from \cite{BeEt} should show that parabolic degeneration is in some sense the associated graded of parabolic restriction; it should be interesting to investigate this relationship in more detail.

\section{Stratum braid group representations }

In this section we will define the stratum braid group and compute the eigenvalues of monodromy for its action on the degenerations of standard modules.

\subsection{Definition of the stratum braid group} In this section we fix a linear subspace $\overline{S} \subseteq V$, and we take $S$ to be the open subset
$$S=\overline{S} \setminus \bigcup_{\substack{H \in \mathcal{A} \\ \overline{S} \nsubseteq H}} H.$$ Then all the points of $S$ have the same stabilizer group
$$G_S=\la r \in R \ | \ \mathrm{fix}_V(r) \supseteq \overline{S} \ra.$$ The main examples we have in mind are when $S$ is a stratum for the action of $G$, or when $S$ consists of non-zero multiples of an eigenvector for some element $g \in G$. As before we set
$$N_S=\{g \in G \ | \ g(S) \subseteq S \},$$ and we choose a base point $b \in S$. We observe that $N_S \subseteq N_G(G_S)$, and that $N_S = N_G(G_S)$ is the normalizer of $G_S$ in $G$ if $S$ is a stratum.

We will write $\pi_1(S)$ for the \emph{fundamental groupoid} of $S$: this is the category whose objects are the points of $S$ and whose morphisms are sets
$$\mathrm{Hom}_{\pi_1(S)}(x,y)=\{\hbox{Homotopy classes of paths in $S$ from $x$ to $y$} \}.$$  We will use the convention that $f_1 f_2$ means first travel along the path $f_1$ and then along the path $f_2$: thus if $f_1 \in \mathrm{Hom}(x,y)$ and $f_2 \in \mathrm{Hom}(y,z)$ then $f_1 f_2 \in \mathrm{Hom}(x,z)$.

The \emph{stratum braid group} $B_S=B(G,V,S,b)$ (we use this terminology even though $S$ might not be all of a stratum) is the set of pairs
$$B(G,V,S,b)=\{(f,n) \ | \ n \in N_S, \ f \in \mathrm{Hom}_{\pi_1(S)}(b,n(b)) \},$$ with group law
$$(f_1,n_1)(f_2,n_2)=(f_1 n_1(f_2), n_1 n_2).$$ Note that $f_2$ is a path from $b$ to $n_2(b)$, so that $n_1(f_2)$ is a path from $n_1(b)$ to $n_1 n_2(b)$ and the composition is well-defined. In the case in which the group $N_S$ acts freely on $S$, $B_S$ is the fundamental group of the quotient $S/N_S$ with base point the orbit of $b$. In particular, we obtain the usual braid group in the special case where $S$ is the big open stratum.

Given an $N_S$-equivariant vector bundle $\mathcal{U}$ on $S$ with an $N_S$-equivariant flat connection, we obtain a (right) representation of $B_S$ on the fiber $\mathcal{U}(b)$ of $\mathcal{U}$ at $b$ in the usual way: if $u \in \mathcal{U}(b)$ and $(f,n) \in B_S$, then we define $u(f,n)$ to be the element of $\mathcal{U}(b)$ obtained by choosing a flat section in a neighborhood of  $b$ with value $u$ at $b$, analytically continuing it along $f$ to $n(b)$, and then transporting its value back to $b$ by applying $n^{-1}$. We obtain a left action by taking inverses,
$$(f,n).v=v(f,n)^{-1}=v(n^{-1}f^{-1},n^{-1}).$$ Explicitly, $(f,n)v$ is obtained by analytic continuation along the path $n^{-1}(f^{-1})$ from $b$ to $n^{-1}(b)$, followed by the application of $n$.

\subsection{Monodromy around $C \subseteq \overline{S}$} Let $C \subseteq \overline{S}$ be the intersection $\overline{S} \cap H$ for some $H \in \mathcal{A}$ with $S \nsubseteq H$, and write $G_C \supseteq G_S$ for the corresponding parabolic subgroup of $G$. Let $N_{S,C}=N_S \cap G_C$ be the normalizer of $S$ in $G_C$. Since $C$ is of codimension $1$ in $\overline S$, the quotient $N_{S,C} / G_S$ acts faithfully on the line $\overline S / C$ hence is cyclic, and by \cite{How, Mur} we have a splitting:
\[
N_{S,C} = G_S \rtimes \la h \ra.
\]
Let $n_{S,C}$ be the order of $h$.

Choose a point $p \in C$ with
$$\mathrm{Stab}_{N_S}(p)=N_{S,C}$$ and let $z_1,z_2,\dots,z_d$ be a system of holomorphic coordinates on $\overline{S}$ (here $d=\mathrm{dim}(S)$) with $z_j(p)=0$ for all $j$, $z_d$ a linear function on $\overline{S}$ and
\[
C = \{w \in \overline{S} \ | \ z_d(w)=0 \}.
\]
We may also assume that $z_1,\dots,z_{d-1}$ are $h$-invariant and that $h$ is chosen so that $h z_d=\zeta^{-1} z_d$ where $\zeta=e^{2 \pi i / n_{S,C}}$ is a primitive $n_{S,C}$-th root of $1$. Let
\[
\frac{\partial}{\partial z_1},\dots,\frac{\partial}{\partial z_d}
\]
be vector fields near $p$ with
$$\frac{\partial}{\partial z_j} z_k=\delta_{jk}.$$

Define the \emph{monodromy around $C$} to be the element $T_C$ of the stratum braid group $B_S$ given by
$$T_C=(f,h)$$ where $f$ is the path from the base point $b=(b_1,\dots,b_d)$ to $h(b)=(b_1,\dots,b_{d-1},e^{2 \pi i /n_{S,C}}b_d)$ given by
$$f(t)=(b_1,\dots,b_{d-1},e^{2 \pi i t/ n_{S,C}}b_d).$$ 

In what follows we distinguish the case in which $S$ is a stratum, when we will consider the
$j$-th graded piece $\mathrm{gr}^j_{I_S}( \CC[V_S] \otimes_{\CC[V]} \Delta_c(\lambda))$ of the degeneration, from the case in which it is not; we make this distinction simply by taking $j = 0$ when $S$ is not a stratum, so that we are considering just the restriction to $S$ of the standard module. For any $j$, this is a vector bundle on $S$ with a connection given by the operators $y_S$ for $y \in \overline{S}$:
$$\nabla_y(f_1 f_2 u)=f_1 \left(\partial_y (f_2)  u - \sum_{r \in R \setminus R_S} c_r \la \alpha_r,y \ra \frac{f_2}{\alpha_r} (r(u)-u) \right) .$$ Thus $T_C$ acts on the fiber $\CC[S^\perp]^j \otimes S^\lambda$ of this vector bundle (again, we take $j=0$ if $S$ is not a stratum). We will compute its eigenvalues in terms of the representation theory of the tower of algebras $\CC G \supseteq \CC G_C \supseteq \CC N_{S,C}$.

\subsection{$c$-functions for subgroups and eigenvalues of monodromy} For $S^\lambda \in \mathrm{Irr}(\CC G)$,  $S^\mu \in \mathrm{Irr}(\CC G_C)$  and $S^\nu \in \mathrm{Irr}(\CC N_{S,C})$ let $S^{\lambda,\mu}$ be the $\mu$-isotypic component of $S^\lambda$ and let $S^{\lambda,\mu,\nu}$ be the $\nu$-isotypic component of $S^{\lambda,\mu}$. Treating the various $c_r$'s as variables, define the set of \emph{eigenvalues of monodromy}
\begin{equation} \label{E def}
E^{\lambda,j}_{S,C}=\{\zeta^k e^{2\pi  i  (c_\mu-c_\nu)/n_{S,C}} \}
\end{equation} where $(k,\mu,\nu)$ runs over all triples $0 \leq k \leq n_{S,C}-1$, $\mu \in \mathrm{Irr}(\CC G_C)$, and $\nu \in \mathrm{Irr}(\CC N_{S,C})$, such that $e^{2 \pi  i  k/ {n_{S,C}}}$ is an eigenvalue of $h$ on $\CC[S^\perp]^j \otimes_\CC S^{\lambda,\mu,\nu}$. Here, in accordance with previous notation, $c_\nu$ is the scalar by which $\sum_{r \in R_S} c_r(1-r)$ acts on $S^\nu$ (note that this sum is actually central in $\CC N_{S,C}$ since $G_S$ is normal in $N_{S,C}$). We note that if one further restricts $S^\nu$ to $G_S$, then $c_\nu$ is also the value of the $c$-function for the subgroup $G_S$ on any irreducible constituent of $S^\nu$.

The elements of $E^{\lambda,j}_{S,C}$ are functions of $c=(c_r)_{r \in R}$; if we specialize $c$ to a particular value we will write $\tau(c)$ for the number obtained by substituting $c$ into the function $\tau \in E^{\lambda,i}_{S,C}$.

In the following theorem, part (b) is a version of the results of \cite{BMR} and \cite{Opd}, which correspond to the special case when $S=V^\circ$ is the big open stratum. Our proof is an adaptation of Opdam's argument to our situation. In case (a), when $S$ is just one dimensional, we can give an explicit formula for the monodromy operator that will be crucial for the proof of Theorem \ref{main}.
\begin{theorem} \label{monodromy eigenvalues}
Suppose $\overline{S}$ is a subspace of $V$ and $S$ is the complement in $\overline{S}$ of the set of hyperplanes that do not contain $S$. Let $j \in \ZZ_{\geq 0}$ and assume $j = 0$ if $S$ is not a stratum.
\begin{enumerate}[(a)]
\item In case $S$ is one-dimensional, we have $C=\{0\}$ and may omit it from the notation, writing $N_S=N_{S,C}$ and $T$ for the monodromy. The $D(S) \rtimes N_S$-modules $\CC[S^\perp]^j \otimes_\CC S^{\lambda}$ decompose as direct sums 
$$\CC[S^\perp]^j \otimes_\CC S^{\lambda}=\bigoplus_{\mu \in \mathrm{Irr}(\CC N_S)} \CC[S^\perp]^j \otimes_\CC S^{\lambda,\mu}.$$ For each irreducible representation $S^\mu$ of $N_S$ we then have
$$T=e^{2 \pi i (c_\lambda-c_\mu)/n_S} h \quad \hbox{on $\CC[S^\perp]^j \otimes_\CC S^{\lambda,\mu}$.}$$
\item The monodromy operator $T_C$ around $C$ acting on $\mathrm{gr}^j_{I_S}(\CC[V_S] \otimes_{\CC[V]} S^{\lambda})$ satisfies
$$\prod_{\tau \in E^{\lambda,i}_{S,C}} (T_C-\tau(c))=0.$$
\end{enumerate}
\end{theorem}
\begin{proof}
We prove (b) first, and then observe that the argument simplifies somewhat under the hypotheses of (a) to give the more precise conclusion there. We fix linear coordinates $z_k$ on $V$ such that $z_1,\dots,z_d$ are coordinates on $S$, and write $\frac{\partial}{\partial z_k}$ for the dual basis of $V$ (thought of also as differential operators on functions). We will abbreviate $A=\CC[S^\perp]^j \otimes S^\lambda$.

For each $\alpha_r$ with $\{\alpha_r=0 \} \cap \overline{S}=C$, we renormalize so that $\alpha_r=z_d$ on $\overline{S}$, rewrite $f \otimes u=F$ and take $y=\frac{\partial}{\partial z_d}$ so that by Lemma \ref{stratum Dunkl}
$$y_S(F)=\frac{\partial F}{\partial z_d}-\frac{1}{z_d} a(F)-b(F),$$ where $a,b$ are the endomorphisms of the fiber given by
$$a(f v)=f \sum_{\substack{r \in R \setminus R_S \\�\{\alpha_r=0\} \cap \overline{S}=C}} c_r (s-1) v \quad \text{and} \quad b(fv)=f\sum_{\substack{r \in R \setminus R_S \\ \{\alpha_r=0\} \cap \overline{S} \neq C }} c_r \la \alpha_r, \frac{\partial}{\partial z_d} \ra \frac{1}{\alpha_r} (1-r) v,$$ for $f \in \CC[S^\perp]^i$ and $v \in S^\lambda$, and we note that $b$ is holomorphic near $p$ since $\alpha_r(p) \neq 0$ for each term in the sum that defines it. The operator $a$ may be rewritten as
\[
a=\sum_{r \in  R_C \setminus R_S} c_r (r-1)=\sum_{r \in R_S} c_r (1-r)-\sum_{r \in R_C} c_r(1-r),
\]
where $R_C=R \cap G_C$ is the set of reflections whose hyperplane contains $C$. Therefore the eigenvalues of $a$ acting on $A$ are the numbers $c_\nu-c_\mu$, where $c_\nu$ is the scalar by which $\sum_{r \in R_S} c_r(1-r)$ acts on $S^\nu$ and $c_\mu$ is the scalar by which $\sum_{r \in R_C} c_r (1-r)$ acts on $S^\mu$, in accordance with our previous notation, and $\mu$ and $\nu$ range over pairs such that $S^{\lambda,\mu,\nu} \neq 0$.

We will prove (b) first under the assumption that for each eigenvalue $\gamma$ of the operator $a$ acting on $A$, the number  $\gamma+m$ is not an eigenvalue for any $m \in \ZZ_{>0}$. This assumption holds for generic choices of the parameters $c$: the eigenvalues of $a$ are linear functions of the parameter $c_r$, and so the condition that the difference of two of them not be an integer holds for all parameters $c_r$ belonging to the complement of a locally finite set of hyperplanes.

By the general theory of flat connections, each point of $S$ has a neighborhood so that the space of sections $f$ of the bundle on the neighborhood satisfying
$$y_S(f)=0 \quad \hbox{for all $y \in S$} $$ is of dimension equal to that of the fiber $A$. We now analyze the behavior of such a function restricted to a single variable. Writing $f(z)$ for the restriction of $f$ to the complex line $z_1=z_2=\dots=z_{d-1}=0$ in $\mathcal{U}$ with coordinate function $z=z_d$, we have
$$z \frac{df}{dz}=a f(z)+z b(z) f(z).$$ By the theory of ordinary differential equations with regular singular points, there exist $\gamma \in \CC$ and coefficients $a_m \in A$ so that
$$f(z)=\sum_{m=0}^\infty z^{\gamma+m} a_m$$ converges uniformly on a compact neighborhood of our base point to a holomorphic function solving the differential equation; moreover as $a_0$ ranges over a basis of $A$ these range over a basis of the solution space. Comparing dimensions it follows that there is a unique flat section restricting to each such $f$.

Expanding
$$b(z)=\sum_{k=0}^\infty z^k b_k$$ for $b_k \in \mathrm{End}_\CC(S^\lambda)$ we have
\begin{equation} \label{n equivariance}
h (z b(z)) h^{-1}=z b(z) \quad \text{and hence} \quad \zeta^{-k-1} h b_k h^{-1}=b_k.
\end{equation}

Now the differential equation is
$$\sum_{m=0}^\infty (\gamma+m) z^{\gamma+m} a_m=\sum_{m=0}^\infty z^{\gamma+m} a a_m+\sum_{m=0}^\infty z^{\gamma+m}  \left( \sum_{k=0}^{m-1} b_k a_{m-k-1} \right)$$ and we obtain the recurrence
$$(\gamma+m-a) a_m=\sum_{k=0}^{m-1} b_k a_{m-k-1}$$ determining $a_m$ for all $m>0$ from $a_0$, provided $(\gamma-a) a_0=0$ and the operators $\gamma+m-a$ are all invertible for $m>0$. Thus a solution is uniquely determined by the eigenvector $a_0$ for $a$, with $\gamma$ the corresponding eigenvalue, using our assumption that $\gamma+m$ is not an eigenvalue for any $m \in \ZZ_{>0}$. If we also choose $a_0$ so that $ha_0=\zeta^j a_0$ for some integer $j$ with $0 \leq j  \leq n_{S,C}-1$ then the recurrence above together with \eqref{n equivariance} implies that
\[
h a_m=\zeta^{m+j} a_m \quad \hbox{for all $m \in \ZZ_{\geq 0}$}
\]
and hence
\[
f(z)=z^{\gamma-j} \sum_{m=0}^\infty z^{m+j} a_m \quad \text{with} \quad  \sum_{m=0}^\infty z^{m+j} a_m \quad \hbox{invariant under $h$}.
\]
Now the monodromy transformation takes $v:=f(b)$ to
$$T_C(v)=hf( e^{-2 \pi i/n_{S,C} }b)=e^{-(\gamma-j) 2\pi i/n_{S,C}}b^{\gamma-j} \sum_{m=0}^\infty b^{m+j} a_m=e^{ 2\pi i(j-\gamma) /n_{S,C}} v,$$ (in the proof of this formula, if $n_{S,C}=1$ then we interpret $f(e^{-2\pi i /n_{S,C}} b)$ to mean the value of the analytic continuation of $f$ at the end of a full loop).

In order to describe the eigenvalues of $a$ on $S^\lambda$ (and hence the monodromy around $C$) we write
$$S^\lambda=\bigoplus_{\substack{ \mu \in \mathrm{Irr} \CC G_C \\ \nu \in \mathrm{Irr} \CC N_{S,C}}} S^{\lambda,\mu,\nu}$$ where $S^{\lambda,\mu,\nu}$ is the $\nu$-isotypic component of the $\mu$-isotypic component of $S^\lambda$. Then the action of $a$ on $S^{\lambda,\mu,\nu}$ is by the scalar
$$a|_{S^{\lambda,\mu,\nu}}=c_\nu-c_\mu.$$ Hence the formula above can be rewritten
$$T_C(v)=e^{ 2\pi i(j+c_\mu-c_\nu) /n_{S,C}} v.$$
Furthermore, as $a_0$ ranges over a basis of $A$, the flat sections $f$ constructed above range over a basis of the space of all flat sections. This proves part (a) for generic parameters $c$, and the general case follows by the argument of the second paragraph of the proof of Theorem 3.6 from \cite{Opd}.

For part (a), we note that $d=1$, take the base point to be $z=z_1=1$, and observe that the operator $b$ above is $b=0$, while the operator $a$ preserves each summand in the direct sum decomposition, proving the first assertion. The flat sections are of the form $f F$ where $f \in \CC[S^\perp]^j$ and $F:S \rightarrow S^\lambda$ is a function satisfying 
$$\frac{\partial F}{\partial z}=\frac{1}{z} a(F).$$ Let $v \in S^{\lambda,\mu}$. Since $a$ acts by the scalar $c_\mu-c_\lambda$ on $S^{\lambda,\mu}$, the function
$$F=z^{c_\mu-c_\lambda} \otimes v$$ solves the differential equation and has $F(1)=v$, so the action of monodromy is by
$$Tv=e^{-2\pi i (c_\mu-c_\lambda)/n_S} h v$$ as claimed.
\end{proof}

\subsection{Proof of Theorem \ref{main}}

This is essentially the same as the proof of Corollary \ref{simplemain}, but taking into account the extra structure that the monodromy representation gives on the fibers. 
We will apply Theorem \ref{monodromy eigenvalues} in the special case where $S=\CC^\times v$ for a non-zero vector $v$.
We recall that $N_S=\{g \in G \ | \ g(S) \subseteq S \}$, and $N_S$ splits as a semi-direct product $G_S \rtimes \la h \ra$ for some $h \in N_S$. The braid group $B_S$ in this case is generated by $G_S$ and the monodromy operator $T$ around $0 \in \overline{S}$ from the previous section, with the relation
$$T g T^{-1} = h g h^{-1}.$$

Given an irreducible $\CC N_S$-module $S^\nu$ and a complex number $\tau$, we define  an irreducible $B_S$-module $S^\nu_\tau$, the \emph{$\tau$-twist of $S^\nu$} by setting $S^\nu_\tau=S^\nu$ as a $G_S$-module, with the action of $T$ defined by $T=e^{2 \pi i \tau} h$. (The stratum braid group is a central extension of $N_S$, and these twisted irreducibles are projective representations of $N_S$).

\cut{
\begin{theorem} \label{necessary condition}
Suppose $v \in V \setminus \{0 \}$ and let $S=\CC^\times v$. If $L_c(\lambda)$ is not supported on $S$, then for each irreducible representation $S^\nu$ of $N_S$ with $S^{\lambda,\nu} \neq 0$, there is an irreducible $\CC G$ module $S^{\mu}$ and an irreducible $\CC N_S$-module $S^\eta$ such that
\begin{enumerate}
\item[(1)] $\mu >_c \lambda$,
\item[(2)] $S^{\mu,\eta} \neq 0$, and
\item[(3)] $S^\nu_{(c_\lambda-c_\nu)/n_S} \cong S^\eta_{(c_\mu-c_\eta)/n_S}$ as $B_S$-modules.
\end{enumerate}
\end{theorem}
\begin{proof}
}
Let $F_S:\OO_c \rightarrow B_S-\mathrm{mod}$ be the functor sending $M \in \OO_c$ to the monodromy representation of $B_S$ on the fiber of the restriction $M|_S$. Then $F$ is right exact, so from the exact sequence
$$0 \longrightarrow \mathrm{Rad}(\Delta_c(\lambda)) \longrightarrow \Delta_c(\lambda) \longrightarrow L_c(\lambda) \longrightarrow 0$$ and the fact that $F_S(L_c(\lambda))=0$, we obtain an exact sequence
$$F_S( \mathrm{Rad}(\Delta_c(\lambda))) \longrightarrow F_S(\Delta_c(\lambda)) \longrightarrow 0.$$ Right exactness of $F_S$ and $\mathrm{Rad}(\Delta_c(\lambda)) \in \OO^{>c_\lambda}_c$  together imply that $F_S(\Delta_c(\lambda))$ belongs to the subcategory of $B_S$-mod consisting of modules possessing filtrations whose layers are quotients of the $F_S(\Delta_c(\mu))$ for $\mu >_c \lambda$. The action of $B_S$ on $F_S(\Delta_c(\lambda))$ is determined by Theorem \ref{monodromy eigenvalues}: $F_S(\Delta_c(\lambda))$ is a direct sum of irreducible representations of $B_S$ of the form $S^\nu_{(c_\lambda-c_\nu)/n_S}$, where $\nu$ runs over the irreducible representations of $\CC N_S$ such that $S^{\lambda,\nu} \neq 0$. 
\cut{
The theorem follows from this.
\end{proof}
}

\begin{remark}
In the case $G$ is a Weyl group, applying Springer correspondence to the lowest weights allowed by Theorem \ref{main} seems to produce distinguished nilpotent orbits. Perhaps this is evidence that some kind of cuspidal character sheaves correspond to finite dimensional Cherednik algebra modules.
\end{remark}

\subsection{Degeneration of standard modules}

\begin{lemma} \label{degen sum}
Assume that $S$ is a stratum of dimension $1$. The $H_c(G_S,S^\perp)$-module $\mathrm{Deg}_{S,p}(\Delta_c(\lambda))$ is a direct sum
$$\mathrm{Deg}_{S,p}(\Delta_c(\lambda))=\bigoplus \Delta_c(S^{\lambda,\mu}) z^{c_\mu-c_\lambda}$$ with $T$-action on the factor indexed by $\mu$ given by
\[
T|_{\Delta_c(S^{\lambda,\mu}) z^{c_\mu-c_\lambda}}=e^{2 \pi i (c_\lambda-c_\mu)/n_S} h.
\]
and with $\Delta_c(S^{\lambda,\mu}) z^{c_\mu-c_\lambda}$ isomorphic to $\Delta_c(S^{\lambda,\mu})$ as a $H_c(G,S^\perp)$-module.
\end{lemma}
\begin{proof}
There is a $\CC[S \times S^\perp] \rtimes N_S$-module isomorphism
$$\CC[S \times S^\perp] \otimes S^\lambda \cong \mathrm{gr}_{I_S}( \CC[V_S] \otimes_{\CC[V]} \Delta_c(\lambda)).$$ Now using Lemma \ref{stratum Dunkl}, on the associated graded the action of the Dunkl operators is given by the formula
$$y_S(f \otimes u)=\partial_y(f) \otimes u -\sum_{r \in R_S} c_r \la \alpha_r,y \ra \frac{f-r(f)}{\alpha_r} \otimes r(u) \quad \hbox{for $y \in S^\perp \subseteq F^{-1}/F^0$.} $$ The crucial point here is that in the associated graded module we can ignore the terms from the formula in Lemma \ref{stratum Dunkl} coming from $r \in R \setminus R_S$. By part (a) of Theorem \ref{monodromy eigenvalues}, the direct sum and $T$ action are as claimed, and our explicit formula for the Dunkl operator proves the last assertion.
\end{proof}

\subsection{Morphisms between standard modules}

Here we restate and prove Theorem \ref{standard recursion} from the introduction, with the notation used there:
\begin{theorem} \label{standard recursion2}
If there is a non-zero morphism $\Delta_c(\lambda) \rightarrow \Delta_c(\mu)$, then for each one-dimensional stratum $S$ there irreducible representations $S^\chi$ and $S^\nu$ of $N_S$ with $ c_\lambda-c_\mu-(c_\chi-c_\nu) \in \ZZ_{\geq 0}$ and a non-zero map of $H_c(N_S,S^\perp)$-modules
$$\Delta_c(S^{\chi} \otimes_\CC L^{\otimes (c_\lambda-c_\mu-(c_\chi-c_\nu))}) \rightarrow \Delta_c(S^{\nu}).$$ 
\end{theorem}
\begin{proof}
Suppose $f:\Delta_c(\lambda) \rightarrow \Delta_c(\mu)$ is a non-zero morphism (necessarily of polynomial degree $c_\lambda-c_\mu$) and fix a one-dimensional stratum $S$. Let $m$ be the greatest integer with $f(S^\lambda) \subseteq I(S)^m S^\mu$, where $I(S)$ is the ideal of $S$ in $\CC[V]$. We define the \emph{initial term} of $f$ with respect to $S$ to be the morphism of polynomial degree $m$
$$\mathrm{in}(f):\mathrm{Deg}_{S,p}(\Delta_c(\lambda)) \rightarrow \mathrm{Deg}_{S,p}(\Delta_c(\mu))$$ induced by $f$; it is non-zero. Using Lemma \ref{degen sum} It follows that there are irreducible representations of $\chi$ and $\nu$ of $\CC N_S$ and a non-zero morphism of degree $m$
$$\Delta_c(S^{\lambda,\chi}) z^{c_\chi-c_\lambda} \rightarrow \Delta_c(S^{\mu,\nu}) z^{c_\nu-c_\lambda}.$$ We have $m=c_\chi-c_\nu$ and hence $$c_\lambda-c_\mu-(c_\chi-c_\nu) \in \ZZ_{\geq 0},$$ proving the first claim. The second follows by twisting the map $\mathrm{in}(f)$ by the appropriate tensor power of $L$. 
\end{proof} We observe that if we choose any irreducible $G_S$-constituents $S^{\lambda_1}$ of $S^\chi$ and $S^{\mu_1}$ of $S^\nu$, then $c_{\lambda_1}=c_\chi$ and $c_{\mu_1}=c_\nu$ so that $$c_\lambda-c_{\chi}-(c_\mu-c_{\nu})=c_\lambda-c_{\lambda_1}-(c_\mu-c_{\mu_1})=c_{U,n}-c_{T,n}$$ for any flag $V_\bullet$ with $\overline{S}$ as its one-dimensional piece and any $T \in \mathrm{SYT}_{V_\bullet}(\lambda)$ and $U \in \mathrm{SYT}_{V_\bullet}(\mu)$ with $T=(\lambda,\lambda^1,\dots)$ and $U=(\mu,\mu^1,\dots)$.
Theorem \ref{standard map1} from the introduction now follows by induction on the dimension of $V$.

For the groups $G(r,1,n)$, and considering only the flag defined by $V_i=\{x_{i+1}=x_{i+2}=\cdots=x_n=0\}$ Theorem \ref{standard recursion2} may be stated non-recursively: see Theorem \ref{gr1n} in the next section.  

\subsection{Highest weight ordering}
In this section we show that, as a formal consequence of our necessary condition for the existence of a map between two standard modules, we may coarsen the $c$-ordering on category $\OO_c(G,V)$ to obtain a new highest weight ordering.

A \emph{highest weight category} is a triple $(C,\Delta,\leq)$, where $C$ is an abelian category and $(\Delta,\leq)$ is a finite poset of objects of $C$ with the properties:
\begin{enumerate}
\item[(a)] Each $\Delta(\lambda)$ is indecomposable of finite length, with a simple top $L(\lambda)$ and a projective cover $P(\lambda)$ in $C$,

\item[(b)] if $L(\lambda)$ appears as a composition factor in the radical of $\Delta(\mu)$ then $\lambda < \mu$, 

\item[(c)] the kernel of the map $P(\lambda) \rightarrow \Delta(\lambda)$ has a finite filtration with layers among the standard objects $\Delta(\mu)$ with $\mu > \lambda$, and
\item[(d)] if $M \in C$ and $\mathrm{Hom}(\Delta(\lambda),M)=0$ for all $\lambda$, then $M=0$.
\end{enumerate} By (c), the multiplicity $[P(\lambda):\Delta(\mu)]$ is well-defined. The highest weight category satisfies \emph{Bernstein-Gelfand-Gelfand reciprocity} (or simply ``BGG reciprocity") if in addition $[P(\lambda):\Delta(\mu)]=[\Delta(\mu):L(\lambda)]$ for all $\lambda,\mu$. 

Given two partial orders $\leq$ and $\leq'$, we say that $\leq'$ \emph{refines} $\leq$ if $$\lambda \leq \mu \implies \lambda \leq' \mu.$$ In this situation we say $\leq$ is \emph{coarser than} $\leq'$. It is an immediate consequence of the definition that if $\leq'$ is a partial order refining $\leq$ then $(C,\Delta,\leq')$ is a highest weight category. The following lemma shows that for a highest weight category with BGG reciprocity, there is a unique coarsest ordering $\leq_{\mathrm{min}}$ on $\Delta$ making $(C,\Delta,\leq_{\mathrm{min}})$ a highest weight category.
\begin{lemma} \label{hwo}
Suppose $(C,\Delta,\leq)$ is a highest weight category with BGG reciprocity and let $\leq_{\mathrm{min}}$ be the coarsest partial ordering such that $\mathrm{Hom}(\Delta(\lambda),\Delta(\mu)) \neq 0 \ \implies \lambda \leq_{\mathrm{min}} \mu$. Then $(C,\Delta,\leq_{\mathrm{min}})$ is a highest weight category.
\end{lemma}
\begin{proof}
This is essentially the same as the proof of Corollary 3.6 from \cite{GoSt}; we include the proof here for the reader's convenience.

Axioms (a) and (d) do not depend on the ordering, so hold for $(C,\Delta,\leq_{\mathrm{min}})$, and in the presence of BGG reciprocity axioms (b) and (c) are equivalent, so it suffices to check (b). This we do by induction on $\mu$. If $\mu$ is minimal for the ordering $\leq_{\mathrm{min}}$ then axiom (d) shows that $\Delta(\mu)=L(\mu)$ is simple so the induction begins. If $[\Delta(\mu):L(\lambda)] \neq 0$ with $\mu \neq \lambda$, then we obtain a non-zero map $P(\lambda) \rightarrow \Delta(\mu)$. This map cannot be surjective, for that would imply that the top $L(\mu)$ of $\Delta(\mu)$ is equal to the top $L(\lambda)$ of $P(\lambda)$. Since $P(\lambda)$ has a $\Delta$-filtration there is a non-zero, non-surjective map $\Delta(\nu) \rightarrow \Delta(\mu)$ for some $\nu$ with $[\Delta(\nu):L(\lambda)]=[P(\lambda):\Delta(\nu)] \neq 0$. Since the map is non-zero, $L(\nu)$ appears as composition factor of $\Delta(\mu)$, and since it is non-surjective $\nu \neq \mu$. By definition we have $\nu \leq_{\mathrm{min}} \mu$; since $\nu \neq \mu$ the inductive hypothesis shows that $\lambda \leq_{\mathrm{min}} \nu$, finishing the proof.
\end{proof}


\section{Examples}

\subsection{$G=G(r,1,n)$} The group $G(r,1,n)$ is generated by the symmetric group $S_n$, acting on $\CC^n$ by permuting the coordinates, and the transformations $\zeta_i$ for $1 \leq i \leq n$ that multiply the $i$th coordinate by the primitive $r$th root of unity $\zeta=e^{2\pi  i /r}$. The classes of the reflections are: those containing $\zeta_1^j$, for $1 \leq j \leq r-1$, and the class containing the transposition $(12) \in S_n$.

The irreducible $\CC G(r,1,n)$-modules are labeled by $r$-tuples $\lambda=(\lambda^0,\lambda^1,\dots,\lambda^{r-1})$ of partitions with $n$ total boxes. The irreducible $S^{\lambda}$ may be constructed by inflating the module $S^{\lambda^0} \otimes S^{\lambda^1} \otimes \cdots \otimes S^{\lambda^{r-1}}$ from the product of symmetric groups $S_{n_0} \times S_{n_1} \times \cdots \times S_{n_{r-1}}$ to $G(r,1,n_0) \times G(r,1,n_1) \times \cdots \times G(r,1,n_{r-1}) \subseteq G(r,1,n)$ by asking all $\zeta_i$'s in the $j$th $G(r,1,n_j)$ to act by $\zeta^j$. We then put
$$S^{\lambda}=\mathrm{Ind}^{G(r,1,n)}_{G(r,1,n_0) \times G(r,1,n_1) \times \cdots \times G(r,1,n_{r-1})} S^{\lambda^0} \otimes \cdots \otimes S^{\lambda^{r-1}}.$$ 

A \emph{standard Young tableau} on $\lambda$ is a filling of its boxes (in other words, of the boxes of all the partitions $\lambda^k$) by the integers $1,2,\dots,n$ in such a way that in each component $\lambda^k$, the entries are strictly increasing left to right and top to bottom. We define the \emph{content} $\mathrm{ct}(b)$ of a box $b$ of $\lambda$ in the usual way, as $\mathrm{ct}(b)=i-j$ if $b$ appears in column $i$ and row $j$ of a component $\lambda^k$, and for each box $b$ of $\lambda^k$ we put $\beta(b)=k$. For each $1 \leq k \leq n$ let $$\phi_j=\sum_{\substack{1 \leq k < j \\ 0 \leq l \leq r-1}} \zeta_j^l s_{jk} \zeta_j^{-l}.$$ The module $S^\lambda$ has a basis $v_T$, indexed by standard Young tableaux $T$ on $\lambda$, with the property that 
$$\phi_j v_T=r \mathrm{ct}(T^{-1}(j)) v_T \quad \text{and} \quad \zeta_j v_T=e^{2 \pi i \beta(T^{-1}(j))/r}.$$ Furthermore, just as in the $S_n$-case, the restriction rule from $G(r,1,n)$ to $G(r,1,n-1)$ may be conveniently described in terms of this basis: for a fixed box $b \in \lambda$, the $\CC$-span of the $v_T$ with $T(b)=n$ is isomorphic to the irreducible module $S^{\lambda \setminus b}$ for $G(r,1,n-1)$ if $\lambda \setminus b$ is an $r$-partition (and is $0$ otherwise).

We write $c_0$ for the parameter $c_s$ for any reflection $s$ in the class $\zeta_k^l s_{jk} \zeta_k^{-l}$, and $c_l$ for the parameter $c_s$ for the class $\zeta_k^l$. If we define
$$d_j=\sum_{1 \leq l \leq r-1} \zeta^{lj} c_l \quad \hbox{for $0 \leq j \leq r-1$}$$ then a quick calculation using this version of Young's orthonormal form for $S^{\lambda^\bullet}$ gives
$$c_{\lambda^\bullet}=r{n \choose 2}c_0+n d_0 - \sum_{b \in \lambda^\bullet} \left(\mathrm{ct}(b)r c_0 + d_{\beta(b)} \right).$$ If we define the \emph{charged content} $c(b)$ of a box $b \in \lambda^\bullet$ by the equation
$$c(b)=\mathrm{ct}(b) rc_0+d_{\beta(b)},$$ then we can rewrite this yet again as
\begin{equation} \label{c function as charged contents}
c_{\lambda^\bullet}=r {n \choose 2} c_0+nd_0-\sum_{b \in \lambda^\bullet} c(b).
\end{equation} Thus, up to sign and a term that is independent of $\lambda^\bullet$, the $c$ function is the sum of the charged contents of the boxes of $\lambda^\bullet$.

Write $\mathrm{SYT}(\lambda)$ for the set of all standard Young tableaux on $\lambda$. Now Theorem \ref{standard recursion} may be applied for the chain $$G(r,1,n) \supseteq G(r,1,n-1) \supseteq G(r,1,n-2) \supseteq \cdots$$ to obtain:
\begin{theorem} \label{gr1n}
If there is a non-zero morphism $\Delta_c(\lambda) \rightarrow \Delta_c(\mu)$ then there are $T \in \mathrm{SYT}(\lambda)$ and $U \in \mathrm{SYT}(U)$ with 
$$c(U^{-1}(i))-c(T^{-1}(i)) \in \ZZ_{\geq 0} \quad \text{and} \quad c(U^{-1}(i))-c(T^{-1}(i))=\beta(U^{-1}(i))-\beta(T^{-1}(i)) \ \mathrm{mod} \ r.$$ 
\end{theorem}

Each maximal parabolic subgroup of $G(r,1,n)$ is conjugate to one corresponding to a stratum of the form $$S=\{x_1=x_2=\cdots=x_k=0 \ \text{and} \ x_{k+1}=x_{k+2}=\dots=x_n\}$$ where $0 \leq k \leq n-1$, in which case $G_S=G(r,1,k) \times S_{n-k}$. In each case the normalizer is a direct product
$$N_S=G_S \times \la \zeta \ra,$$ where $\zeta=e^{2 \pi  i /r}$ is the scalar matrix generating the center of $G(r,1,n)$. So the irreducible representations of $N_S$ may be indexed by pairs $(M,\zeta^j)$ corresponding to an irreducible representation $M$ of $G_S$ and an integer $0 \leq j \leq r-1$, with $\zeta$ acting by $\zeta^j$ on $M$.

Consider first the case $k=n-1$: restriction from $G(r,1,n)$ to $G(r,1,n-1) \times S_1=G(r,1,n-1)$. Given an $r$-tuple of partitions $\lambda^\bullet=(\lambda^0,\lambda^1,\dots,\lambda^{r-1})$, the restriction of the irreducible $G(r,1,n)$-module $S^{\lambda^\bullet}$ to $G(r,1,n-1)$ is
$$\mathrm{res}^{G(r,1,n)}_{G(r,1,n-1)} S^{\lambda^\bullet}=\bigoplus_{\nu^\bullet=\lambda^\bullet \setminus b} S^{\nu^\bullet}$$ where the sum is over all removable boxes $b$ of $\lambda^\bullet$ (a box is removable if it is a removable box of $\lambda^i$ for some $i$). Each summand is also an irreducible $N_S=G(r,1,n-1) \times \la \zeta \ra$-module, with $\zeta$ acting by the scalar
$$\zeta \longmapsto \zeta^{\sum_{b \in \lambda^\bullet} \beta(b)}.$$ Thus if $S=\{(0,0,\dots,0,x) \} \nsubseteq \mathrm{supp}(L_c(\lambda^\bullet))$ then Theorem \ref{main} implies that for each removable box $b$ of $\lambda^\bullet$, there is some $\mu^\bullet >_c \lambda^\bullet$ with a removable box $b' \in \mu^\bullet$ such that $\lambda^\bullet \setminus b=\mu^\bullet \setminus b'=\nu^\bullet$ and
$$\zeta^{\sum_{b \in \lambda^\bullet} \beta(b)} e^{2 \pi  i  (c_{\lambda^\bullet}-c_{\nu^\bullet})/r}=\zeta^{\sum_{b \in \mu^\bullet} \beta(b)} e^{2 \pi  i  (c_{\mu^\bullet}-c_{\nu^\bullet})/r}.$$ Using the formula \eqref{c function as charged contents} we obtain
\begin{corollary}
If $L_c(\lambda^\bullet)$ is not supported on $S=\{(0,0,\dots,0,x) \}$ then for every removable box $b \in \lambda^\bullet$, there is an addable box $b'$ with
\begin{enumerate}
\item[(1)] $c(b)-c(b') \in \ZZ_{>0}$, and
\item[(2)] $c(b)-c(b')=\beta(b)-\beta(b')$ mod $r$.
\end{enumerate}
\end{corollary}

For general $k$, the restriction from $G(r,1,n)$ to $G(1,1,k) \times G(r,1,n-k)$ involves the Littlewood-Richardson coefficients and their generalizations to $r$-tuples of partitions. We do not go into details here; a consequence is that if one $d_i$ is much bigger than the rest then for $L_c(\lambda^\bullet)$ to be finite dimensional we must have $\lambda^i=\emptyset$.

\subsection{The spherical irreducible $L_c(\mathrm{triv})$.} Etingof's paper \cite{Eti} classifies when $L_c(\mathrm{triv})$ is finite dimensional for real reflection groups. Here we discuss the case of the groups $G(r,1,n)$. We begin with the following lemma, valid for an arbitrary reflection group $G$.
\begin{lemma}
Let $S \subseteq V$ be a stratum with corresponding parabolic subgroup $G_S$, and let $V'$ be the orthogonal complement to $\overline{S}$ in $V$. If $L_c(\mathrm{triv},G_S,V')$ is finite dimensional then the support of $L_c(\mathrm{triv})$ for $H_c(G,V)$ is contained in $\bigcup_{g \in G} g(S)$.
\end{lemma}
\begin{proof}
This is a corollary of Lemma 2.9 from \cite{BGS}.
\end{proof} The lemma may be used to prove that $L_c(\mathrm{triv})$ is finite dimensional in a number of cases. For $G=G(r,1,n)$ it gives
\begin{corollary} \label{suff condition}
For $G=G(r,1,n)$, suppose that either
\begin{enumerate}
\item[(a)] $r (n-1) c_0+d_0-d_j \in j+r \ZZ_{\geq 0}$ for some $1 \leq j \leq r-1$ or
\item[(b)] the following two conditions hold: $c_0=\ell/m$ with $2 \leq m \leq n$ and $(\ell,m)=1$ for positive integers $\ell$ and $m$ with $m$ a divisor of $n$, and furthermore
$$r (p-1) c_0+d_0-d_j \in j+r \ZZ_{\geq 0} \quad \hbox{for some $(n/m-1)m+1 \leq p \leq n$ and $1 \leq j \leq r-1$.}$$
\end{enumerate} Then $L=L_c(\mathrm{triv})$ is finite dimensional.
\end{corollary} \begin{proof}
With the assumption in part (a), this follows from \cite{Gri}, Theorem 7.5, or from \cite{EtMo}. In (b), our assumption on $c_0$, the classification of finite dimensional modules for $S_n$, and the lemma above imply that the support of $L$ is contained in the set of points whose coordinates may be divided up into $n/m$ subsets of size $m$, with the coordinates belonging to a given subset all equal to one another. As above, \cite{Gri}, Theorem 7.5 and the lemma above then combine to show that if 
$$r (p-1) c_0+d_0-d_j \in j+r \ZZ_{\geq 0} \quad \hbox{for some $n/m \leq p \leq n$ and $1 \leq j \leq r-1$}$$ then the support of $L$ is contained in the set of points with at least $p$ zeros among their coordinates. Putting these two facts together proves the corollary.
\end{proof} In case $r=2$, the corollary produces all $c$'s for which $L_c(\mathrm{triv})$ can be finite dimensional. 
\begin{question} 
If $L_c(\mathrm{triv})$ is finite dimensional, do the conditions given in Corollary \ref{suff condition} hold for $r>2$? 
\end{question}

\subsection{$G=G(2,1,2)=W(B_2)$}
 This example is of course well known; we work it out in detail here to illustrate the methods we have applied later on for the exceptional groups. Since there is an involution of the Coxeter diagram, we need to distinguish between long and short roots (the parametrization of the characters depends on this). We will denote the parabolic subgroup generated by a simple reflection associated with the long (resp. short) simple root by $A_1$ (resp. $\widetilde A_1$). In each case, the normalizer of the parabolic subgroup is the product of the parabolic subgroup with the center $\{1, w_0\}$, where $w_0 = -1$. The tables below give the restrictions of the simple $\CC G$-modules to each maximal parabolic subgroup. Since we also give the action of $w_0$, one can deduce the restriction to the normalizer.
 
Since there are two conjugacy classes of reflections, the parameter space is two dimensional. We denote by $c_1$ (resp. $c_2$)  the parameter of the class of reflections corresponding to short (resp. long) roots.
For a given $c = (c_1, c_2)$, let $\prec_c$ be the order on $\Lambda$ defined by $\lambda \prec_c \mu$ if
$c_\mu - c_\lambda \in \RR_{>0}$ (rather than $\ZZ_{>0}$ as in the definition of $<_c$).
We have ordered the rows of the tables below so that if $\mu \succ_c \lambda$, then the $\mu$ row is higher than the $\lambda$ row. Since the order depends on $c$, we have to make a table for each possible $\prec_c$-order. In general, one can define an equivalence relation on the space of $c$ parameters, such that two parameters are equivalent if they give rise the same partial order $\prec_c$. We call the equivalence classes \emph{regions}. For a region $D$, we denote by $\prec_D$ the common order $\prec_c$ for $c\in D$. We may restrict to parameters $c$, with $c_1$, $c_2 > 0$. For $G = W(B_2)$, there are three regions in the upper right quadrant: $D^- := \{c_1 < c_2\}$, $D^0 := \{c_1 = c_2\}$, and $D^+ := \{c_1 > c_2\}$. 

Let us explain in detail how to read off from the tables what our criterion says.
For a given simple module $L_c(\lambda)$, the central element $\sum_{s\in R}c_r(1-r)$ acts by a scalar which is
a linear function $a_1(\lambda) c_1 + a_2(\lambda) c_2$ of the parameters. The coefficients $a_1(\lambda)$ and $a_2(\lambda)$ will be displayed in the first column, for the generic regions $D^-$ and $D^+$. On the critical line $D^0$, the relevant information is rather the sum $a(\lambda) = a_1(\lambda) + a_2(\lambda)$. 
The second column gives the integer $m(\lambda) \in \{0,1\}$ such that the central element $w_0$ acts by $(-1)^{m(\lambda)}$ on $S^\lambda$.
The next column indicates the bipartition $\lambda$ indexing the simple $\CC G$-module (the two parts of $\lambda$ are separated by a dot). The table then gives the restriction of $S^\lambda$ to each parabolic subgroup. 

Fix a region $D$.
By Corollary \ref{simplemain}, if $L_c(\lambda)$ is finite dimensional for some value of $c\in D$, then for each column $\nu$ such that the entry $(\lambda, \nu)$ of the table is non-zero, there must exist a row $\mu$, with $\mu\succ_{D}\lambda$, such that the entry $(\mu, \nu)$ of the table is also non-zero.
We highlighted the maximal (for the $\prec_D$ order) non-zero entries in each column. If there is any highlighted entry in the row $\lambda$, then the module $L_c(\lambda)$ is infinite dimensional for all $c\in D$.

This first step, using only $\prec_D$, discards many simple modules $S^\lambda$. For the ones that are not discarded, using $<_c$ will give conditions on $c\in D$ for $L_c(\lambda)$ to be finite dimensional. Assume $\dim L_c(\lambda) < \infty$. So there is no highlighted entry in row $\lambda$: For each non-zero entry in the row, say in column $\nu$ for some parabolic subgroup, there has to be at least one $\mu$ such that $\mu \succ_D \lambda$. But the Theorem says more: for at least one of those $\mu$'s, we must also have
\[
(a_1(\mu) - a_1(\lambda))c_1 + (a_2(\mu) - a_2(\lambda))c_2 \in 2\ZZ + m(\lambda).
\] The examples below illustrate how this is applied.

\subsubsection{$D^- = \{c_1 < c_2\}$}

\[
\begin{array}{cccc|cc|cc|}
\cline{5-8}
&&&&\multicolumn{2}{c|}{A_1}&\multicolumn{2}{c|}{\raisebox{1.25em}{}\widetilde A_1}
\\
\hline \multicolumn{1}{|c}{a_1}&
\multicolumn{1}{c|}{a_2}&
\multicolumn{1}{c|}{\text{sgn}}&
\multicolumn{1}{c|}{\lambda\backslash \nu}&
11&2&11&2
\\
\hline
\multicolumn{1}{|c}{4}&
\multicolumn{1}{c|}{4}&
\multicolumn{1}{c|}{0}&
\multicolumn{1}{c|}{.11}&
\highlight{1}&.&\highlight{1}&.\\
\multicolumn{1}{|c}{0}&
\multicolumn{1}{c|}{4}&
\multicolumn{1}{c|}{0}&
\multicolumn{1}{c|}{11.}&
1&.&.&\highlight{1}\\
\multicolumn{1}{|c}{2}&
\multicolumn{1}{c|}{2}&
\multicolumn{1}{c|}{1}&
\multicolumn{1}{c|}{1.1}&
1&\highlight{1}&1&1\\
\multicolumn{1}{|c}{4}&
\multicolumn{1}{c|}{0}&
\multicolumn{1}{c|}{0}&
\multicolumn{1}{c|}{.2}&
.&1&1&.\\
\multicolumn{1}{|c}{0}&
\multicolumn{1}{c|}{0}&
\multicolumn{1}{c|}{0}&
\multicolumn{1}{c|}{2.}&
.&1&.&1\\
\hline
\end{array}
\]

Looking at the highlighted entries, we see that the modules $L_c(\lambda)$, for $\lambda\in\{.11, 1.1, 11.\}$, are infinite dimensional for all $c\in D^-$. So the only possible finite dimensional modules are $L_c(2.)$ and $L_c(.2)$.

Now assume $L_c(2.)$ is finite dimensional for some $c\in D^-$. The column labelled $2$ for $A_1$ gives:
\[
c_1 + c_2 \in \ZZ + \frac 1 2 \text{ or } c_1 \in \frac 1 2 \ZZ
\]
while the column labelled $2$ for $\widetilde A_1$ gives:
\[
c_1 + c_2 \in \ZZ + \frac 1 2 \text{ or } c_2 \in \frac 1 2 \ZZ.
\]
Since both have to be satisfied, our criterion is:
\[
c_1 + c_2 \in \ZZ + \frac 1 2 \text{ or } \left(c_1 \in \frac 1 2 \ZZ \text{ and } c_2 \in \frac 1 2 \ZZ\right).
\]
So we find a union of lines and points. Comparing with \cite{Chm} (see also \cite{Eti}), we see that our necessary condition is sufficient for the lines, while among the remaining points we are not able to discard the points with both coordinates in $\ZZ$.

Next, assume $L_c(.2)$ is finite dimensional for some $c\in D^-$. The column labelled $2$ for $A_1$ gives:
\[
-c_1 + c_2 \in \ZZ + \frac 1 2
\]
while the column labelled $11$ for $\widetilde A_1$ gives:
\[
-c_1 + c_2 \in \ZZ + \frac 1 2 \text{ or } c_2 \in \frac 1 2 \ZZ.
\]
Since the first condition implies the second one, we can forget about the latter. Therefore we get
\[
-c_1 + c_2 \in \ZZ + \frac 1 2
\]
which, by \cite{Chm}, is also sufficient.


\subsubsection{$D^0 = \{c_1 = c_2\}$}

\[
\begin{array}{ccc|cc|cc|}
\cline{4-7}
&&&\multicolumn{2}{c|}{A_1}&\multicolumn{2}{c|}{\raisebox{1.25em}{}\widetilde A_1}
\\
\hline \multicolumn{1}{|c|}{a}&
\multicolumn{1}{c|}{\text{sgn}}&
\multicolumn{1}{c|}{\lambda\backslash \nu}&
11&2&11&2
\\
\hline
\multicolumn{1}{|c|}{8}&
\multicolumn{1}{c|}{0}&
\multicolumn{1}{c|}{.11}&
\highlight{1}&.&\highlight{1}&.\\
\multicolumn{1}{|c|}{4}&
\multicolumn{1}{c|}{0}&
\multicolumn{1}{c|}{11.}&
1&.&.&\highlight{1}\\
\multicolumn{1}{|c|}{4}&
\multicolumn{1}{c|}{1}&
\multicolumn{1}{c|}{1.1}&
1&\highlight{1}&1&\highlight{1}\\
\multicolumn{1}{|c|}{4}&
\multicolumn{1}{c|}{0}&
\multicolumn{1}{c|}{.2}&
.&\highlight{1}&1&.\\
\multicolumn{1}{|c|}{0}&
\multicolumn{1}{c|}{0}&
\multicolumn{1}{c|}{2.}&
.&1&.&1\\
\hline
\end{array}
\]

If $c = (c_1, c_1)$, then the only possible finite dimensional simple $H_c$-module is $L_c(2.)$ (all the other rows have some highlighted entry).
Looking at the $2$ column for $A_1$, we see that $4c_1$ has to be an integer, either even ($.2$ row) or odd
($1.1$ row). The $2$ column for $\widetilde A_1$ gives the same condition.

\subsubsection{$D^+ = \{c_1 > c_2\}$}

\[
\begin{array}{cccc|cc|cc|}
\cline{5-8}
&&&&\multicolumn{2}{c|}{A_1}&\multicolumn{2}{c|}{\raisebox{1.25em}{}\widetilde A_1}
\\
\hline \multicolumn{1}{|c}{a_1}&
\multicolumn{1}{c|}{a_2}&
\multicolumn{1}{c|}{\text{sgn}}&
\multicolumn{1}{c|}{\lambda\backslash\nu}&
11&2&11&2
\\
\hline
\multicolumn{1}{|c}{4}&
\multicolumn{1}{c|}{4}&
\multicolumn{1}{c|}{0}&
\multicolumn{1}{c|}{.11}&
\highlight{1}&.&\highlight{1}&.\\
\multicolumn{1}{|c}{4}&
\multicolumn{1}{c|}{0}&
\multicolumn{1}{c|}{0}&
\multicolumn{1}{c|}{.2}&
.&\highlight{1}&1&.\\
\multicolumn{1}{|c}{2}&
\multicolumn{1}{c|}{2}&
\multicolumn{1}{c|}{1}&
\multicolumn{1}{c|}{1.1}&
1&1&1&\highlight{1}\\
\multicolumn{1}{|c}{0}&
\multicolumn{1}{c|}{4}&
\multicolumn{1}{c|}{0}&
\multicolumn{1}{c|}{11.}&
1&.&.&1\\
\multicolumn{1}{|c}{0}&
\multicolumn{1}{c|}{0}&
\multicolumn{1}{c|}{0}&
\multicolumn{1}{c|}{2.}&
.&1&.&1\\
\hline
\end{array}
\]

Here the only possible finite dimensional simple $H_c$-modules are $L_c(2.)$ and $L_c(11.)$. This case is actually symmetric to the $D^-$-case, via the diagram automorphism. So for $c \in D^+$, our criterion says:
\[
L_c(2.) \text{ finite dimensional } \Longrightarrow c_1 + c_2 \in \ZZ + \frac 1 2 \text{ or } \left(c_1 \in \frac 1 2 \ZZ \text{ and } c_2 \in \frac 1 2 \ZZ\right)
\]
\[
L_c(11.) \text{ finite dimensional } \Longrightarrow c_1 - c_2 \in \ZZ + \frac 1 2
\]

We remark that, although the $\prec_D$ order depends on $D$, for a given simple $\CC G$-module we only care about its relative position with respect to the others. The trivial module will always be in the bottom row (since we assume $c_1, c_2 > 0$), so for the trivial module the division into three regions is irrelevant. For a general $G$, for each character there will be a unions of regions which give the same criterion (see also the case of $F_4$).


\subsection{Exceptional real groups} The next six sections deal with the exceptional real groups. We leave the rank $2$ groups aside, as they are to be discussed in the thesis of the second author. We note that if the center of a real group is of order $2$, then for each maximal parabolic $G_S$, the normalizer $N_S$ is a direct product $N_S=G_S \times \{ \pm 1\}$. This is the case in all of the following examples except $G=W(E_6)$, and makes the application of Theorem \ref{main} especially straightforward.

We write $V$ for the reflection representation of $G$ throughout, and when no more familiar description is available, we will use the GAP notation $\phi_{x,y}$ for an irreducible $G$-module of dimension $x$ appearing in degree $y$ of the coinvariant algebra but not in lower degree. In the following lists, for groups $G$ with only one conjugacy class of reflections (all but $W(F_4)$) we assume that the corresponding parameter $c$ is positive. We assumed moreover that if $c$ is constant its denominator divides one of the degrees of $G$, otherwise, as shown in \cite{DJO}, category $\mathcal{O}_c$ would be semisimple. Any representation not appearing in the list for a given group can not be finite dimensional for any choice of $c>0$.

We have tried to list previously known results where they are available.

\subsection{$G=W(H_3)$}
\begin{enumerate}
\item[(1)] $L_c(\mathrm{triv})$ can be finite dimensional only if $10c$ is an odd positive integer or $6c$ is an integer.

\item[(2)] $L_c(V)$ and $L_c(\widetilde{V})$ can be finite dimensional only if $c$ is half an odd positive integer, where $\widetilde{V}$ is the Galois conjugate of $V$ by $\sqrt{5} \mapsto -\sqrt{5}$ .
\end{enumerate} In Theorem 3.1 of \cite{BaPu}, Balagovic and Puranik proved that $L_c(\mathrm{triv})$ is finite dimensional if and only if the denominator of $c$ is $10$, $6$, or $2$, while $L_c(V)$ and $L_c(\widetilde{V})$ are finite dimensional exactly when $c$ has denominator $2$.

\subsection{$G=W(H_4)$}
\begin{enumerate}[(1)]
\item If $L_c(\mathrm{triv})$ is finite dimensional, then the denominator of $c$ divides a degree and our criterion does not give any further restriction (Etingof \cite{Eti} proved that such a condition is also sufficient provided $c$ is not an integer).
\item If $L_c(V)$ or $L_c(\widetilde{V})$ is finite dimensional, then $10c \in \ZZ_{>0}$ or $15c \in \ZZ_{>0}$ or $6c \in \ZZ_{>0}$. Here $\widetilde{V}$ is the Galois conjugate of $V$.  According to \cite{Nor}, $L_{\frac{1}{d}}(V)$ is finite dimensional if $d\in\{3,5,6,10\}$ and, conjecturally, if $d=2$, while  $L_{\frac{1}{d}}(\widetilde{V})$ is finite dimensional if $d\in\{3,6,10,15\}$ and, conjecturally, if $d=2$. 
\item If $L_c(\phi_{9,2})$ or $L_c(\phi_{9,6})$ is finite dimensional, then $10c \in \ZZ_{>0}$ or $4c \in \ZZ_{>0}$. These are Galois conjugates. By \cite{Nor}, $L_{\frac{1}{d}}(\phi_{9,2})$ is finite dimensional if $d\in\{5,4\}$, while  $L_{\frac{1}{d}}(\widetilde{V})$ is finite dimensional if $d\in\{4,10\}$. Moreover, it is conjectured in \cite{Nor} that if $c=1/2$, then $L_c(\phi_{9,2})$ and $L_c(\phi_{9,6})$  have finite dimension. 
\item If $L_c(\phi_{16,3})$  is finite dimensional, then $3c \in 2\ZZ_{\geq 0} + 1$ or $5c\in 2\ZZ_{>0}$. In fact, by \cite{Nor} $L_{\frac{1}{3}}(\phi_{16,3})$ is finite dimensional. 
\item If $L_c(\phi_{16,6})$ is finite dimensional, then $3c \in 2\ZZ_{>0}$ or $5c \in \ 2\ZZ_{\geq 0} + 1$. According to \cite{Nor} $L_{\frac{1}{5}}(\phi_{16,3})$ is finite dimensional. 
\item If $L_c(\phi_{25,4})$ is finite dimensional, then $2c \in  2\ZZ_{\geq 0} + 1$.  In \cite{Nor} it is conjectured that $L_{\frac{1}{2}}(\phi_{25,4})$ is finite dimensional.
\end{enumerate}

\subsection{$G=W(F_4)$}

We assume $c_1, c_2 > 0$ (we recover the other cases by twisting by a linear character) and we set $\kappa = c_2 / c_1$. The critical slopes $\kappa$ where the $c$-order changes are:
1/5, 1/4, 1/3, 2/5, 1/2, 2/3, 1, 3/2, 2, 5/2, 3, 4, 5. However, if we fix one simple $G$-module, only part of these slopes will be relevant, i.e. the slopes where the relative position of this simple with respect to the others actually changes. This defines zones (angular sectors) in the space of parameters, or intervals in the $\kappa$ line. For example, the trivial module is always below all other modules, so in the spherical case there is only one zone to consider
($0< \kappa< +\infty$). We may also use the diagram symmetry: for a pair of characters exchanged by the symmetry, we only need to consider one of them; and for a character fixed by the symmetry, we may assume $c_1 \geq c_2$, i.e. $\kappa \leq 1$. 

Since, in a given zone, the potential cuspidality only depends on the class of $(c_1, c_2)$ modulo $\ZZ^2$, it is more convenient to work with the parameters $q_j = e^{2i\pi c_j}$, $j = 1$, $2$. We will describe the locus of ``potential cuspidality'' in the $(q_1, q_2)$ plane (given by our criterion) by its defining ideal (for the reduced structure).  We used {\sf GAP3} \cite{GAP3} and {\sf MAGMA} \cite{MAGMA}.

The characters of $G$ for which there may be cuspidal modules for some values of $\kappa$ are:
$\phi_{1,0}$, $\phi_{1,12}'$, $\phi_{1,12}''$, $\phi_{2,4}'$, $\phi_{2,4}''$, $\phi_{4,1}$,
$\phi_{4,7}'$, $\phi_{4,7}''$, $\phi_{8,3}'$, $\phi_{8,3}''$, $\phi_{9,2}$. That's 11 out of 25, but there are really only 7 cases to consider, taking the diagram symmetry into account. We will give more details for the spherical case.

\subsubsection{$\phi_{1,0}$}

As we said before, in the spherical case there is only one zone: for all the slopes above, the position of $\phi_{1,0}$
relatively to other simple modules does not change (it is always at the bottom).
The locus of ``potential cuspidality'' has 19 components over $\QQ$ (which are the Galois orbits of the components over the algebraic closure): four one-dimensional and 15 zero-dimensional. Let $\Phi_j$ denote the $j$-th cyclotomic polynomial. The ideal of potential cuspidality is:

\[
\begin{array}{l}
\Phi_6(q_1q_2) \Phi_4(q_1q_2) \Phi_2(q_1q_2) \Phi_1(q_1q_2)\\
(\Phi_3(q_1), \Phi_1(q_2)) (\Phi_6(q_1), \Phi_2(q_2)) (q_1 - q_2, \Phi_3(q_2)) (q_1 - q_2, \Phi_6(q_2))\\
(\Phi_2(q_1), \Phi_6(q_2)) (\Phi_1(q_1), \Phi_3(q_2)) (\Phi_{12}(q_1), \Phi_2(q_2)) (\Phi_3(q_1), \Phi_4(q_2))\\
(q_1^2 - q_2 - 1, \Phi_3(q_2)) (\Phi_4(q_1), \Phi_3(q_2)) (q_1 - q_2^2 + 1, \Phi_{12}(q_2)) (\Phi_1(q_1), \Phi_{12}(q_2))\\
(q_1 + q_2^4, \Phi_{18}(q_2)) (q_1 - q_2^4 + q_2, \Phi_{18}(q_2)) (q_1 - q_2, \Phi_{18}(q_2))
\end{array}
\]

The first row gives the four one-dimensional components. In terms of $c$ parameters, those are the lines
$c_1 + c_2 = m/6$ with $m\in\ZZ$ not divisible by $3$, and the lines $c_1 + c_2 = m/4$ with $m \in \ZZ$.
Comparing with \cite{Eti}, we see that only the factor $\Phi_1(q_1q_2)$ is superfluous,
i.e. the lines $c_1 + c_2 = m$ with $m \in\ZZ$.

The rest gives the 15 (Galois orbits of) zero-dimensional components. For example, the first one is:
$q_1$ is a primitive third root of unity, and $q_2 = 1$; i.e. $c_1 = m/3$ with $m\in\ZZ$ not divisible $3$, and
$c_2\in\ZZ$. Comparing with \cite{Eti}, some are superfluous, while some do not appear because they are contained in the superfluous line mentioned above.

We can summarize the condition given by the necessary condition with the following picture. This is a $1 \times 1$ square in the $(c_1, c_2)$ plane (since $c_1$ and $c_2$ only matter modulo $1$). Green lines and dots
are those given by \cite{Eti} (i.e. the true answer). Our criterion is quite good for lines, since there is
only one irrelevant line ($c_1 + c_2 = 1$, or more generally $c_1 + c_2 \in \ZZ$). But there are many red dots!

In Etingof's notation:
\begin{itemize}
\item case 2a does not appear since this point is on the red line;
\item 2b is half on the red line, and the other points are $(q_1 - q_2, \Phi_3(q_2))$, i.e. $(1/3, 1/3)$ and $(2/3, 2/3)$;
\item 2c is both $(q_1^2 - q_2 - 1, \Phi_3(q_2))$ and $(q_1 - q_2^2 + 1, \Phi_{12}(q_2))$ (note that the Groebner basis algorithm breaks the symmetry);
\item 2d is $(q_1 - q_2, \Phi_6(q_2))$, i.e. $(1/6, 1/6)$ and $(5/6, 5/6)$
\end{itemize}

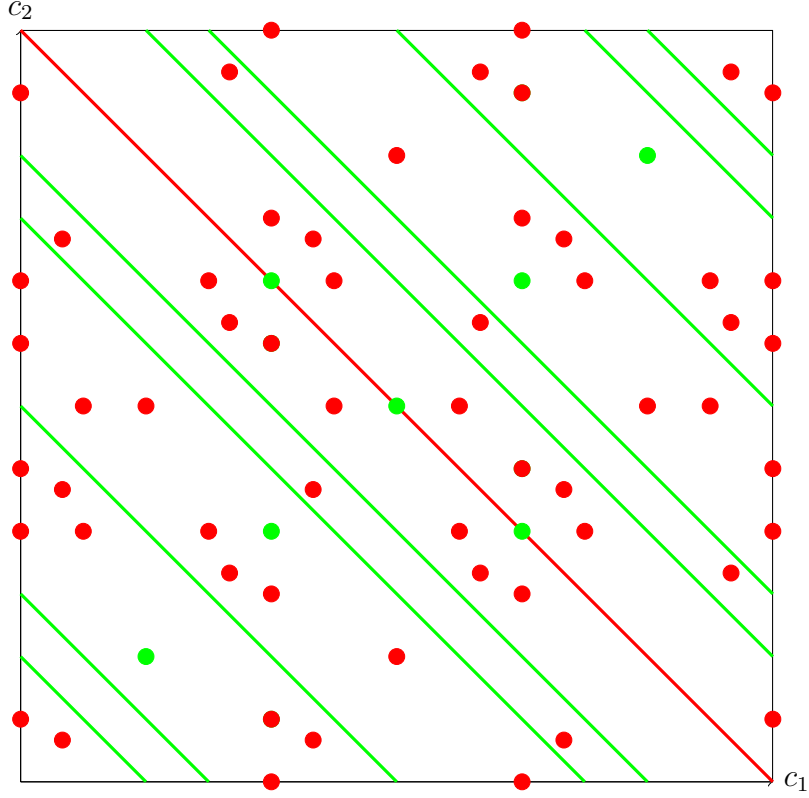
\begin{center}
\begin{tikzpicture}[scale=10]
\tikzstyle{axes}=[]
\tikzstyle{wall}=[thick]
\tikzstyle{relevant wall}=[very thick]
\tikzstyle{dot}=[fill]
\draw (0,0) grid (1,1);
\begin{scope}[style=axes]
\draw[->] (0,0) -- (1,0) node[right] {$c_1$} coordinate(c1 axis);
\draw[->] (0,0) -- (0,1) node[above] {$c_2$} coordinate(c2 axis);
\end{scope}

\begin{scope}[very thick,green,auto=left]
\draw (1/4, 0) -- (0, 1/4);
\draw (1/6, 0) -- (0, 1/6);
\draw (1/2, 0) -- (0, 1/2);
\draw (5/6, 0) -- (0, 5/6);
\draw (3/4, 0) -- (0, 3/4);
\draw[red] (1, 0) -- (0, 1);
\draw (1/4, 1) -- (1, 1/4);
\draw (1/6, 1) -- (1, 1/6);
\draw (1/2, 1) -- (1, 1/2);
\draw (5/6, 1) -- (1, 5/6);
\draw (3/4, 1) -- (1, 3/4);
\end{scope}

\begin{scope}[green]
\draw[fill] (1/2, 1/2) circle (.3pt);
\draw[fill] (1/3, 1/12) circle (.3pt);
\draw[fill] (1/3, 7/12) circle (.3pt);
\draw[fill] (2/3, 5/12) circle (.3pt);
\draw[fill] (2/3, 11/12) circle (.3pt);
\draw[fill] (1/6, 1/6) circle (.3pt);
\draw[fill] (5/6, 5/6) circle (.3pt);
\draw[fill] (1/3, 1/3) circle (.3pt);
\draw[fill] (1/3, 2/3) circle (.3pt);
\draw[fill] (2/3, 1/3) circle (.3pt);
\draw[fill] (2/3, 2/3) circle (.3pt);
\end{scope}

\begin{scope}[red]
\draw[fill] (1/3, 0) circle (.3pt);
\draw[fill] (1/3, 1) circle (.3pt);
\draw[fill] (2/3, 0) circle (.3pt);
\draw[fill] (2/3, 1) circle (.3pt);
\draw[fill] (1/6, 1/2) circle (.3pt);
\draw[fill] (5/6, 1/2) circle (.3pt);
\draw[fill] (1/2, 1/6) circle (.3pt);
\draw[fill] (1/2, 5/6) circle (.3pt);
\draw[fill] (0, 1/3) circle (.3pt);
\draw[fill] (1, 1/3) circle (.3pt);
\draw[fill] (0, 2/3) circle (.3pt);
\draw[fill] (1, 2/3) circle (.3pt);
\draw[fill] (1/12, 1/2) circle (.3pt);
\draw[fill] (5/12, 1/2) circle (.3pt);
\draw[fill] (7/12, 1/2) circle (.3pt);
\draw[fill] (11/12, 1/2) circle (.3pt);
\draw[fill] (1/3, 1/4) circle (.3pt);
\draw[fill] (1/3, 3/4) circle (.3pt);
\draw[fill] (2/3, 1/4) circle (.3pt);
\draw[fill] (2/3, 3/4) circle (.3pt);
\draw[fill] (1/12, 1/3) circle (.3pt);
\draw[fill] (5/12, 2/3) circle (.3pt);
\draw[fill] (7/12, 1/3) circle (.3pt);
\draw[fill] (11/12, 2/3) circle (.3pt);
\draw[fill] (1/3, 1/12) circle (.3pt);
\draw[fill] (2/3, 5/12) circle (.3pt);
\draw[fill] (1/3, 7/12) circle (.3pt);
\draw[fill] (2/3, 11/12) circle (.3pt);
\draw[fill] (1/4, 1/3) circle (.3pt);
\draw[fill] (3/4, 1/3) circle (.3pt);
\draw[fill] (1/4, 2/3) circle (.3pt);
\draw[fill] (3/4, 2/3) circle (.3pt);
\draw[fill] (0, 1/12) circle (.3pt);
\draw[fill] (0, 5/12) circle (.3pt);
\draw[fill] (0, 7/12) circle (.3pt);
\draw[fill] (0, 11/12) circle (.3pt);
\draw[fill] (1, 1/12) circle (.3pt);
\draw[fill] (1, 5/12) circle (.3pt);
\draw[fill] (1, 7/12) circle (.3pt);
\draw[fill] (1, 11/12) circle (.3pt);
\end{scope}

\begin{scope}[red]
\draw[fill] (13/18, 1/18) circle (.3pt);
\draw[fill] (11/18, 5/18) circle (.3pt);
\draw[fill] (1/18, 7/18) circle (.3pt);
\draw[fill] (17/18, 11/18) circle (.3pt);
\draw[fill] (7/18, 13/18) circle (.3pt);
\draw[fill] (5/18, 17/18) circle (.3pt);

\draw[fill] (1/18, 13/18) circle (.3pt);
\draw[fill] (5/18, 11/18) circle (.3pt);
\draw[fill] (7/18, 1/18) circle (.3pt);
\draw[fill] (11/18, 17/18) circle (.3pt);
\draw[fill] (13/18, 7/18) circle (.3pt);
\draw[fill] (17/18, 5/18) circle (.3pt);

\draw[fill] (1/18, 1/18) circle (.3pt);
\draw[fill] (5/18, 5/18) circle (.3pt);
\draw[fill] (7/18, 7/18) circle (.3pt);
\draw[fill] (11/18, 11/18) circle (.3pt);
\draw[fill] (13/18, 13/18) circle (.3pt);
\draw[fill] (17/18, 17/18) circle (.3pt);
\end{scope}

\end{tikzpicture}
\end{center}

\subsubsection{$\phi_{1,12}'$}

There are three zones in this case. The locus of potential cuspidality is given by the following table.

\[
\begin{array}{|c|c|}
\hline
0 < \kappa \leq 1 & 1\\
\hline
1 < \kappa \leq 2 & \Phi_6(q_1q_2^{-1}) \Phi_4(q_1q_2^{-1}) \Phi_2(q_1q_2^{-1}) \Phi_1(q_1q_2^{-1})\\
\hline
 &
\Phi_6(q_1q_2^{-1}) \Phi_4(q_1q_2^{-1}) \Phi_2(q_1q_2^{-1}) \Phi_1(q_1q_2^{-1})\\
& (\Phi_6(q_1), \Phi_2(q_2)) (q_1 + q_2 - 1, \phi_6(q_2)) (\Phi_2(q_1), \phi_6(q_2))\\
2 < \kappa& (\Phi_{12}(q_1), \Phi_1(q_2)) (q_1^2 + q_2, \Phi_3(q_2)) (\Phi_4(q_1), \Phi_3(q_2))\\
& (q_1 + q_2^5 - q_2^2, \Phi_{18}(q_2)) (q_1 - q_2^5, \Phi_{18}(q_2)) (q_1 + q_2^2, \Phi_{18}(q_2))\\
\hline
\end{array}
\]

Note that the variety defined by $(1)$ is empty: there cannot be any cuspidal modules when $0<\kappa \leq 1$.

One can deduce the results for $\phi_{1,12}''$ via the diagram automorphism.

%
%
%
%
%
%
%
%

\subsubsection{$\phi_{2,4}'$}

\[
\begin{array}{|c|c|}
\hline
0 < \kappa \leq 1/2 & \Phi_6(q_2) \Phi_2(q_2)\\
\hline
& \Phi_6(q_2) \Phi_2(q_2)\\
& (\Phi_2(q_1), \Phi_1(q_2)) (\Phi_1(q_1), \Phi_1(q_2)) (q_1 + q_2 + 1, \Phi_3(q_2)) (q_1 - q_2 - 1, \Phi_3(q_2))\\
1/2 < \kappa \leq 1 & (\Phi_2(q_1), \Phi_4(q_2)) (\Phi_1(q_1), \Phi_4(q_2)) (q_1 + q_2^2, \Phi_{12}(q_2)) (q_1 + q_2^2, \Phi_8(q_2))\\
& (q_1 - q_2^2, \Phi_{12}(q_2)) (q_1 - q_2^2, \Phi_8(q_2)) (q_1 + q_2^2, \Phi_{18}(q_2)) (q_1 - q_2^2, \Phi_{18}(q_2))\\
\hline
& (\Phi_2(q_1), \Phi_1(q_2)) (\Phi_1(q_1), \Phi_1(q_2)) (q_1 + q_2 + 1, \Phi_3(q_2)) (q_1 - q_2 - 1, \Phi_3(q_2))\\
& (\Phi_2(q_1), \Phi_4(q_2)) (\Phi_1(q_1), \Phi_4(q_2)) (q_1 + q_2^2, \Phi_{12}(q_2)) (q_1 + q_2^2, \Phi_8(q_2))\\
1 < \kappa & (q_1 - q_2^2, \Phi_{12}(q_2)) (q_1 - q_2^2, \Phi_8(q_2)) (q_1 + q_2^2, \Phi_{18}(q_2)) (q_1 - q_2^2, \Phi_{18}(q_2))\\
& (\Phi_6(q_1), \Phi_1(q_2)) (q_1 + q_2, \Phi_3(q_2)) (\Phi_2(q_1), \Phi_3(q_2))\\
& (q_1 + q_2^4 - q_2, \Phi_{18}(q_2)) (q_1 - q_2^4 - q_2, \Phi_9(q_2))\\
\hline
\end{array}
\]

One can deduce the results for $\phi_{2,4}''$ via the diagram automorphism.

In \cite{Nor}, Norton deals with the equal parameter case ($\kappa=1$ in our notation). She proves that $L_{(c,c)}(\phi_{2,4}')$ and $L_{(c,c)}(\phi_{2,4}'')$ are finite dimensional for $c=\frac{1}{6}$ and conjectures  they are finite dimensional for $c=\frac{1}{2}$, which is compatible with our results.

\subsubsection{$\phi_{4,1}$}

\[
\begin{array}{|c|c|}
\hline
& \Phi_2(q_1 q_2)\\
0 < \kappa < 1 & 
(\Phi_1(q_1), \Phi_1(q_2))
(\Phi_3(q_1), \Phi_1(q_2))
(\Phi_6(q_1), \Phi_1(q_2))
(q_1 + q_2 + 1, \Phi_3(q_2))\\
& (q_1 + q_2, \Phi_3(q_2))
(q_1 - q_2, \Phi_3(q_2))
(\Phi_2(q_1), \Phi_3(q_2))
(\Phi_1(q_1), \Phi_3(q_2))\\
\hline
\kappa = 1 & \Phi_3(q)\Phi_4(q)
\\
\hline
\end{array}
\]

The ideal for $\kappa > 1$ is obtained from the one for $\kappa < 1$ by swapping $q_1$ and $q_2$.

 As for $\kappa=1$, we intersected the potential cuspidality locus with the diagonal $q_1=q_2$ and denoted by $q$ this common value. From our condition we deduce that if $L_{(c,c)}(\phi_{2,4}')$ is finite dimensional, then $4c\in\mathbb{Z}_{>0}$ or $3c\in\mathbb{Z}_{>0}$. By \cite{Nor}, $L_{(c,c)}(\phi_{2,4}')$ is finite dimensional if and only if the denominator of $c$ is exactly $4$ or $3$.


\subsubsection{$\phi_{4,7}'$}

\[
\begin{array}{|c|c|}
\hline
0 < \kappa \leq 1 & 1\\
\hline
1 < \kappa & \Phi_2(q_1 q_2^{-1})\\
\hline
\end{array}
\]

One can deduce the results for $\phi_{4,7}''$ via the diagram automorphism.

\subsubsection{$\phi_{8,3}'$}

\[
\begin{array}{|c|c|}
\hline
0 < \kappa \leq 2 & 1\\
\hline
2 < \kappa & (\Phi_2(q_1), \Phi_2(q_2)) (\Phi_1(q_1), \Phi_2(q_2))
(\Phi_4(q_1), \Phi_1(q_2)) (q_1 - q_2 + 1, \Phi_6(q_2))\\
\hline
\end{array}
\]

One can deduce the results for $\phi_{8,3}''$ via the diagram automorphism.

\subsubsection{$\phi_{9,2}$}

\[
\begin{array}{|c|c|}
\hline
0 < \kappa \leq 1/2 & 1\\
\hline
1/2 < \kappa < 2 & (\Phi_2(q_1), \Phi_2(q_2)) (q_1 + q_2 - 1, \Phi_6(q_2))\\
\hline
2 < \kappa & 1\\
\hline
\end{array}
\]

From the above table, we see that in the case of equal parameters ($\kappa=1$) the module $L_{(c,c)}(\phi_{9,2})$ may be finite dimensional only if $2c\in \mathbb{Z}_{>0}$. This is compatible with \cite{Nor}, where it is conjectured that  $L_{(\frac{1}{2},\frac{1}{2})}(\phi_{9,2})$ has finite dimension.

\vspace{1cm}

\begin{center}%
    \begin{figure} \setlength{\unitlength}{1.20cm}%
\scalebox{1.18}{  \begin{picture}(8.75,6.85)
\put(8.65,-0.3){\makebox(0,0){\small $c_1$}}
 \put(0,0){\vector(1,0){8.75}}
\put(-0.3,6.5){\makebox(0,0){\small $c_2$}}
 \put(0,0){\vector(0,1){6.5}}
\thicklines

 \put(0,0){\color{purple}\line(2,1){7.7}}
\put(8.4,4.2){\makebox(0,0){\rotatebox{27}{\color{purple}\tiny{$c_2=\frac{1}{2}c_1$}}}}

 \put(0,0){\color{green}\line(1,1){5.5}}
\put(6,6){\color{green}\makebox(0,0){\rotatebox{45}{\tiny{$c_2=c_1$}}}}

 \put(0,0){\color{red}\line(1,2){2.8}}
\put(3.1,6.25){\makebox(0,0){\color{red}\rotatebox{65}{\tiny{$c_2\!=2 c_1$}}}}

\put(.7,6.6){\makebox(0,0){$\small\phi_{1,0}$}}
\put(.7,6.0){\makebox(0,0){$\small\color{green}\phi_{1,12}'$}}
\put(.7,5.4){\makebox(0,0){$\small\phi_{2,4}'$}}
\put(.7,4.8){\makebox(0,0){$\small\phi_{2,4}''$}}
\put(.7,4.2){\makebox(0,0){$\small\phi_{4,1}$}}
\put(.7,3.6){\makebox(0,0){$\small\color{green}\phi_{4,7}'$}}
\put(.7,3.0){\makebox(0,0){$\small\color{red}\phi_{8,3}'$}}

\put(4.5,6.6){\makebox(0,0){$\small\phi_{1,0}$}}
\put(4.1,6.0){\makebox(0,0){$\small\color{green}\phi_{1,12}'$}}
\put(3.7,5.4){\makebox(0,0){$\small\phi_{2,4}'$}}
\put(3.3,4.8){\makebox(0,0){$\small\phi_{2,4}''$}}
\put(2.9,4.2){\makebox(0,0){$\small\phi_{4,1}$}}
\put(2.5,3.6){\makebox(0,0){$\small\color{green}\phi_{4,7}'$}}
\put(2.1,3.0){\makebox(0,0){$\small\color{red}\phi_{9,2}$}}

\put(6.6, 4.5){\makebox(0,0){$\small\phi_{1,0}$}}
\put(6.0, 4.1){\makebox(0,0){$\small\color{green}\phi_{1,12}''$}}
\put(5.4, 3.7){\makebox(0,0){$\small\phi_{2,4}'$}}
\put(4.8, 3.3){\makebox(0,0){$\small\phi_{2,4}''$}}
\put(4.2, 2.9){\makebox(0,0){$\small\phi_{4,1}$}}
\put(3.6, 2.5){\makebox(0,0){$\small\color{green}\phi_{4,7}''$}}
\put(3.0, 2.1){\makebox(0,0){$\small\color{purple}\phi_{9,2}$}}

\put(9, .7){\makebox(0,0){$\small\phi_{1,0}$}}
\put(8, .7){\makebox(0,0){$\small\color{green}\phi_{1,12}''$}}
\put(7, .7){\makebox(0,0){$\small\phi_{2,4}'$}}
\put(6, .7){\makebox(0,0){$\small\phi_{2,4}''$}}
\put(5, .7){\makebox(0,0){$\small\phi_{4,1}$}}
\put(4, .7){\makebox(0,0){$\small\color{green}\phi_{4,7}''$}}
\put(3, .7){\makebox(0,0){$\small\color{purple}\phi_{8,3}''$}}

    \end{picture}
    }\caption{$F_4$: Potential finite dimensional simples in each region}
\end{figure}
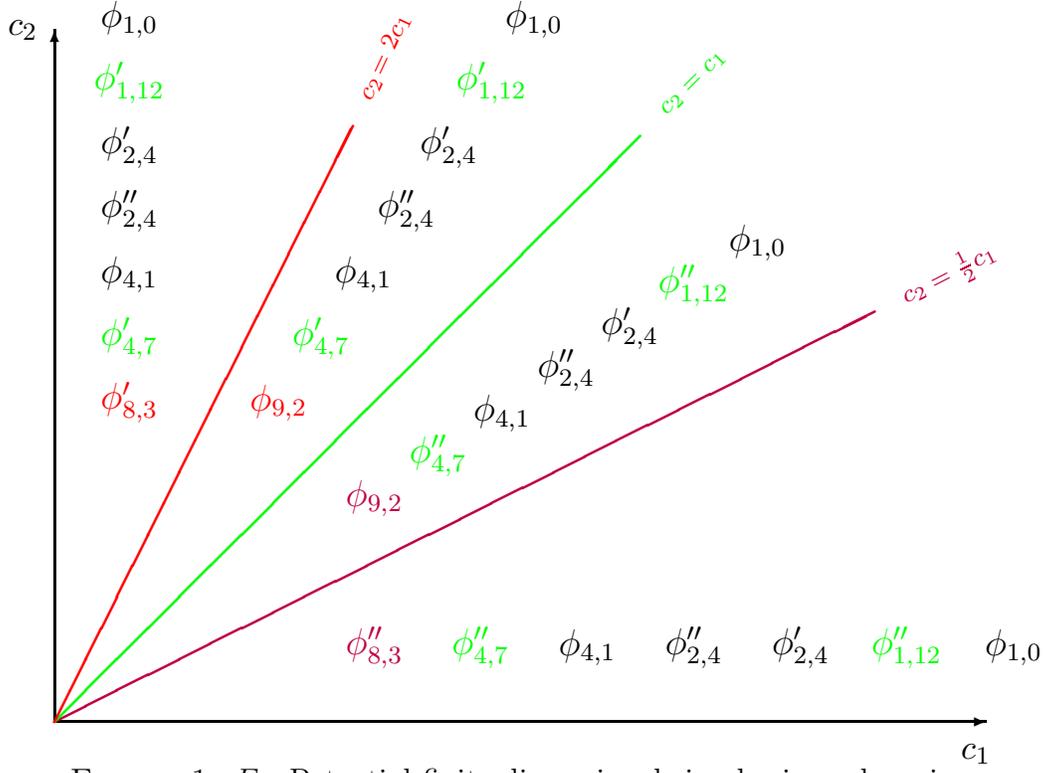
\end{center}


%
%
%
%
%
%
%

\subsection{$G=W(E_6)$}
\begin{enumerate}
\item[(1)] If $L_c(\mathrm{triv})$ is finite dimensional then $12c \in \ZZ_{>0}$ or $9c \in \ZZ_{>0}$. According to \cite{Eti} and \cite{VaVa}, $L_c(\mathrm{triv})$ is finite dimensional if and only if the denominator of $c$ belongs to the set of elliptic numbers $\{3,6,9,12 \}$.
\item[(2)] If $L_c(V)$ is finite dimensional then $6c \in \ZZ_{>0}$. 
\item[(3)] If $L_c(\wedge^2 V)$ is finite dimensional then $3c \in \ZZ_{>0}$.
\end{enumerate}
 By \cite{Nor} $L_c(V)$  is finite dimensional if and only if the denominator of $c$ is exactly $6$ or $3$  and $L_c(\wedge^2 V)$ is finite dimensional if and only if $c$ has denominator $3$.

\subsection{$G=W(E_7)$}
\begin{enumerate}
\item[(1)] If $L_c(\mathrm{triv})$ is finite dimensional then  $14c \in \ZZ_{>0}$ or  $18c \in 2\ZZ_{\geq 0}+1$. By \cite{Eti} and \cite{VaVa}, $L_c(\mathrm{triv})$ is finite dimensional if and only if the denominator of $c$ is an elliptic number: $2,6,14$ or $18$.
\item[(2)] If $L_c(V)$ is finite dimensional, then $6c \in \ZZ_{>0}$ or $10c \in 2\ZZ_{\geq 0}+1$. 
\item[(3)] If $L_c(\phi_{15,7})$ is finite dimensional, then  $2c \in \ZZ_{>0}$ or $6c \in 2\ZZ_{\geq 0}+1$.
\item[(4)] If $L_c(\phi_{21,6})$ is finite dimensional, then $6c \in \ZZ_{>0}$.
\item[(5)] If $L_c(\phi_{27,2})$ or $L_c(\phi_{35,13})$ or $L_c(\phi_{189,5})$ is finite dimensional then $2c \in 2\ZZ_{\geq 0}+1$.
\end{enumerate}
According to \cite{Nor}, $L_{\frac{1}{d}}(V)$ is finite dimensional if $d= 6$ or $10$, while $L_{\frac{1}{d}}(\phi_{15,7})$ and  $L_{\frac{1}{d}}(\phi_{21,6})$  are finite dimensional if $d= 6$.

\subsection{$G=W(E_8)$}
\begin{enumerate}
\item[(1)] If $L_c(\mathrm{triv})$ is finite dimensional, then $30c \in \ZZ_{>0}$ or $24c \in \ZZ_{>0}$ or $20c \in \ZZ_{>0}$. By \cite{Eti} and \cite{VaVa}, $L_c(\mathrm{triv})$ is finite dimensional if and only if the denominator of $c$ is an elliptic number, that is, belongs to the set $\{2,3,4,5,6,8,10,12,15,20,24,30\}$. These are precisely the integers greater than $1$ that divide one of $20,24$, or $30$.
\item[(2)] If $L_c(V)$ is finite dimensional, then $18c \in 2\ZZ_{\geq 0}+1$ or $30c \in \ZZ_{>0}$.
\item[(3)] If $L_c(\phi_{28,8})$ is finite dimensional, then $10c \in \ZZ_{>0}$ or $12c \in \ZZ_{>0}$ or $18c \in 2\ZZ_{\geq 0}+1$.
\item[(4)] If $L_c(\phi_{35,2})$ or $L_c(\phi_{50,8})$ is finite dimensional, then $12 c \in \ZZ_{>0}$.
\item[(5)] If $L_c(\phi_{56,19})$ is finite dimensional, then $6c \in 2\ZZ_{\geq 0}+1$ or $5 c \in \ZZ_{>0}$.
\item[(6)] If $L_c(\phi_{160,7})$ is finite dimensional, then  $3c \in \ZZ_{>0}$ or $8c \in 2\ZZ_{\geq 0}+1$.
\item[(7)] If $L_c(\phi_{175,12})$ or  $L_c(\phi_{300,8})$ is finite dimensional, then $6c \in \ZZ_{>0}$.
\item[(8)] If $L_c(\phi_{210,4})$  is finite dimensional, then $4c\in \ZZ_{>0}$ or $6c \in 2\ZZ_{\geq 0}+1$.
\item[(9)] If $L_c(\phi_{350,14})$ is finite dimensional, then $4c \in 2\ZZ_{\geq 0}+1$.
\item[(10)] If $L_c(\phi_{840,13})$ is finite dimensional, then $3c \in \ZZ_{>0}$.
\item[(11)] If $L_c(\phi_{560,5})$ or $L_c(\phi_{840,14})$ or $L_c(\phi_{1050,10})$ or  $L_c(\phi_{1400,8})$  then $2c \in 2\ZZ_{\geq 0}+1$.
\end{enumerate}
In this case, applying the stronger version of our necessary condition not only provides a restriction  on the set of the parameters $c$ for which a given module can be of  finite dimension, but also allows us to eliminate the module $L_c(\phi_{1575, 10})$, for any $c$, which could not be discarded by using the weaker version.

\subsection{Some exceptional complex examples}  For illustration, we discuss three examples of complex reflection groups for which we have applied the strong version of our criterion for $c$ constant and positive. 
As in the previous sections, we list only modules which can be finite dimensional; any irreducible not appearing in our lists cannot be the lowest weight of a finite dimensional module.

\subsection{$G=G_{4}$} 
\begin{enumerate}
\item[(1)] If $L_c(\textrm{triv})$ is finite dimensional, then the denominator of $c$ has to be $6$, $4$ or $2$. 
\item[(2)] If $L_c(\phi_{2,3})$ or $L_c(\phi_{2,1})$  is finite dimensional, then $2c\in2\mathbb{Z}_{\geq 0}+1$.
\end{enumerate}

\subsection{$G=G_{12}$} 
\begin{enumerate}
\item[(1)] If $L_c(\textrm{triv})$ is finite dimensional, then $8c\in \mathbb{Z}_{> 0}$ or $12c\in\mathbb{Z}_{>0}$. In \cite{BaPo} it is shown that $L_c(\textrm{triv})$ is finite dimensional if $c=\frac{m}{12}$, with $m \equiv 1, 3, 4, 5, 6, 7, 8, 9, 11 \mod 12$.
\item[(2)] If $L_c(\phi_{2,4})$  is finite dimensional, then  $2c\in2\mathbb{Z}_{\geq 0}+1$. According to \cite{BaPo}, this condition is also sufficient.
\item[(3)] If $L_c(\phi_{2,5})$ or $L_c(\phi_{2,1})$ or $L_c(\phi_{4,3})$  is finite dimensional, then $4c\in 2\mathbb{Z}_{\geq 0}+1$ . In \cite{BaPo} is proved that  $L_c(\phi_{2,5})$, resp.  $L_c(\phi_{2,1})$,  is finite dimensional if $c=\frac{m}{4}$, with $m\equiv 1,3 \mod 8$, resp. $m\equiv 5,7 \mod 8$,  while $L_c(\phi_{4,3})$ is always infinite dimensional.
\item[(4) ] If $L_c(\phi_{3,2})$ is finite dimensional, then $4c\in \mathbb{Z}_{> 0}$. By \cite{BaPo}, $L_c(\phi_{3,2})$ is finite dimensional if $4c\in 2\mathbb{Z}_{\geq 0}+1$.
\end{enumerate}

\subsection{$G=G_{24}$} 
\begin{enumerate}
\item[(1)] If $L_c(\textrm{triv})$  is finite dimensional, then $21c\in\mathbb{Z}_{>0}$ or $c=\frac{m}{18}$ with $m\equiv 1,3,5,6,7,9,11,12,13,15,17\mod 18$.
\item[(2)] If $L_c(\phi_{3,3})$ or $L_c(\phi_{3,1})$   is finite dimensional, then $7c\in\mathbb{Z}_{>0}$ or $2c\in 2\mathbb{Z}_{\geq 0}+1$.
\item[(3)] If $L_c(\phi_{6,2})$ is finite dimensional, then $7c\in\mathbb{Z}_{>0}$ or $4c\in 2\mathbb{Z}_{\geq 0}+1$.
\item[(4)] If $L_c(\phi_{7,3})$ is finite dimensional, then $3c\in\mathbb{Z}_{>0}$.
\end{enumerate}

\def\cprime{$'$} \def\cprime{$'$}

\end{document}